\documentclass[svgnames]{amsart}
\usepackage[english]{babel}
\usepackage[utf8]{inputenc}

\usepackage{svg}
\usepackage{hyperref}
\hypersetup{colorlinks,citecolor=Black,linkcolor=Black,urlcolor=Blue}

\usepackage{tikz-cd}
\usepackage{tabularx, ifthen}
\usepackage{amssymb, amsfonts, amsmath, amsthm, mathabx}
\usepackage{epigraph}
\usepackage{enumitem}
\usepackage{comment}

\theoremstyle{plain}
\newtheorem{thm}{Theorem}[section] 
\newtheorem{lem}[thm]{Lemma}
\newtheorem{cor}[thm]{Corollary}
\newtheorem{thmintro}{Theorem}

\newtheorem{claim}[thm]{Claim}

\theoremstyle{definition}
\newtheorem{defn}[thm]{Definition}
\newtheorem{rem}[thm]{Remark}

\newtheorem{notation}[thm]{Notation}

\newtheorem{problintro}[thmintro]{Problem}

\newcommand{\frakS}{\mathfrak{S}}
\newcommand{\diam}{\mathrm{diam}}
\newcommand{\C}{\mathcal C}
\newcommand{\dist}{\mathrm{d}}
\newcommand{\nest}{\sqsubseteq}
\newcommand{\propnest}{\sqsubsetneq}
\newcommand{\orth}{\bot}
\newcommand{\transverse}{\pitchfork}

\newcommand{\MCG}{\mathcal{MCG}^\pm}
\newcommand{\Z}{\mathbb{Z}}
\newcommand{\M}{\mathbb{M}}

\newcommand{\gate}{\mathfrak{g}}

\newcommand{\Cay}[2]{\operatorname{Cay}\left(#1,#2\right)}

\newcommand{\ov}{\overline}

\author[G. Mangioni]{Giorgio Mangioni}
    \address{(Giorgio Mangioni) Department of Mathematics, Heriot-Watt University and Maxwell Institute for Mathematical Sciences, Edinburgh, UK}
    \email{gm2070@hw.ac.uk}

\title[Combination for HQC and geometric subgroups]{A combination theorem for hierarchically quasiconvex subgroups, and application to geometric subgroups of mapping class groups}

\begin{document}

\begin{abstract}
We provide sufficient conditions for two subgroups of a hierarchically hyperbolic group to generate an amalgamated free product over their intersection. The result applies in particular to certain geometric subgroups of mapping class groups of finite-type surfaces, that is, those subgroups coming from the embeddings of closed subsurfaces.

In the second half of the paper, we study under which hypotheses our amalgamation procedure preserves several notions of convexity, such as hierarchical quasiconvexity (as introduced by Behrstock, Hagen, and Sisto) and strong quasiconvexity (every quasigeodesic with endpoints on the subset lies in a uniform neighbourhood). This answers a question of Russell, Spriano, and Tran.
\end{abstract}

\maketitle

\epigraph{I will achieve in my life - Heaven grant that it be not long - some gigantic amalgamation between the two discrepancies so hideously apparent to me.}{Virginia Woolf}

\setcounter{tocdepth}{1}
\tableofcontents

\section*{Introduction}\footnote{Keywords: Mapping class groups, hierarchically hyperbolic groups, quasiconvex subgroups\\MSC classification: 20F65 (Primary) 57K20, 51F30 (Secondary)}
Given a group $G$ and two subgroups $A,B\le G$, it is natural to ask what the subgroup $\langle A, B\rangle_G$ generated by $A$ and $B$ looks like, and in particular if it is isomorphic to the \emph{amalgamated free product} $A *_{C} B.$ In this paper, we prove an amalgamation theorem for when $A$ and $B$ are subgroups of any group $G$ acting “nicely” on a \emph{hierarchical space} (see Definition~\ref{defn:Hierarchical_space} below). This class includes all relatively hyperbolic groups, and all \emph{hierarchically hyperbolic groups} (HHG) in the sense of Behrstock, Hagen, and Sisto \cite{HHS_II} (such as mapping class groups of finite-type surfaces, many $3$-manifold groups, many Coxeter and Artin groups, compact special groups\ldots). We give here a special case of the result, postponing the full statement to Section~\ref{sec:amalgamation}:

\begin{thmintro}\label{thmintro:amalgamation}
    Let $(G,\frakS)$ be a hierarchically hyperbolic group, let $A, B\le G$ be subgroups and let $C=A\cap B$. Suppose that there exists a constant $M\ge 0$ and a domain $Y_a\in \frakS$ for every $a\in (A\cup B)-C$, such that the following hold:
    \begin{enumerate}[label=(\Roman*)]
        \item $\max\left\{\diam_{Y_a}(Cx_0), \diam_{Y_a}(aCx_0)\right\}\le M/10$;
        \item $\dist_{Y_a}(Cx_0, aCx_0)\ge M$; 
        \item If $a\in A-C$ and $b\in B-C$, $Y_a\transverse aY_{b}$;
        \item In the same setting, $\dist_{Y_a}(Cx_0, bCx_0)\le M/10$.
    \end{enumerate}
    There exists $M_0\ge 0$, depending only on $(G,\frakS)$, such that, if $M\ge M_0$, then 
    $$\langle A,B\rangle_G\cong A *_C B.$$
\end{thmintro}

\subsection*{Amalgamation of geometric subgroups}
In understanding the above theorem, one should have in mind the following example. Let $G=\MCG(S)$ be the (extended) mapping class groups of a finite-type surface $S$. Let $U$, $V$ be two closed, connected, incompressible subsurfaces, such that no connected component of $S-U$ (resp. $S-V$) is an annulus. These conditions ensure that $\MCG(U)$ and $\MCG(V)$ naturally embed in $G$, and we denote the image of such embeddings as \emph{geometric embedded} subgroups\footnote{The terminology ``geometric'' is due to Paris and Rolfsen \cite{paris_rolfsen} and denotes the image of the homomorphism $\MCG(W)\to \MCG(S)$ for any closed subsurface $W$, without further assumptions. Our notation thus denotes the cases when a geometric subgroup is also embedded. If there is a standard name for such subgroups we would be grateful for a reference.}. Let $\Gamma$ be the collection of all curves which belong to the boundary of both $U$ and $V$; this multicurve might be empty even if $\partial U $ and $\partial V$ intersect. Suppose $\partial U-\Gamma$ and $\partial V-\Gamma$ are “sufficiently entangled” in the complement of $\Gamma$, meaning that they are far enough in the curve graph $\C (S-\Gamma)$. This condition ensures that the intersection of $\MCG(U)$ and $\MCG(V)$ is the Dehn Twist flat $\Z^{|\Gamma|}$ supported on $\Gamma$. Using the separability of the latter in both $\MCG(U)$ and $\MCG(V)$, one can find finite-index subgroups $A\le \MCG(U)$ and $B\le \MCG(V)$, such that every element $a\in A-\Z^{|\Gamma|}$ acts with large translation length on some subsurface $Y_a\nest U$ (Assumption (II)). Furthermore, the entanglement of the boundaries of $U$ and $V$ ensures that, whenever $a\in A-\Z^{|\Gamma|}$ and $b\in B-\Z^{|\Gamma|}$, the subsurfaces $Y_a$ and $Y_b$ must overlap (Assumption (III)). The above example is analysed more thoroughly in Subsection~\ref{sec:mcg}, where we prove the following, slightly more general result:

\begin{thmintro}\label{thmintro:geometric}
   Let $S$ be a connected finite-type surface, and let $\MCG(U)$ and $\MCG(V)$ be two geometric embedded subgroups, where each of $U$ and $V$ is either connected or a multicurve. Let $\Gamma$ be the collection of all curves which belong to the boundary of both $U$ and $V$ (where, with a little abuse of notation, the boundary of a multicurve denotes its support). Suppose that $\partial U-\Gamma$, $\partial V-\Gamma$ are both non-empty, and that $$\dist_{\C (S-\Gamma)}(\partial U-\Gamma,\partial V-\Gamma)\ge 4.$$

   Then there exist finite index subgroups $A\le \MCG(U)$, $B\le \MCG(V)$, intersecting along the Dehn twist flat $\Z^{|\Gamma|}$, such that
   $$\langle A,B\rangle_{\MCG(S)}\cong A *_{\Z^{|\Gamma|}} B.$$
\end{thmintro}

\subsubsection*{Comparison with the literature}
Our Theorem~\ref{thmintro:geometric} is similar in spirit to Leininger-Reid's result for Veech subgroups along a common multitwist \cite{Leininger_reid}. The main difference is that, while the subgroups there are supported on the whole surface $S$ (and indeed every element which does not lie in the intersection is pseudo-Anosov), our result deals with reducible subgroups, and the large translations of the elements are witnessed by pairs of transverse subsurfaces. 

Our theorem also covers the case when $U$ and $V$ are multicurves. In this setting, it should be compared to Loa's result about free product of multitwists supported on “far enough” multicurves \cite{Loa}. While our procedure requires to pass to finite-index subgroups with large translation in the annular domains, theirs applies to the whole Dehn Twist flats supported on the multicurves, but only when the intersection $U\cap V$ is empty. We also stress that Loa's result gives more information about the amalgam, including the fact that it is undistorted in $\MCG(S)$ and \emph{parabolically geometrically finite}, in the sense of \cite{ddhs}.

However, we point out that our result about multicurves is just a very special case of a theorem of Koberda, which can be used to produce more general, undistorted RAAGs in $\MCG(S)$ (see \cite{koberda_raags} and its quantitative version by Runnels \cite{runnels_effective_gen}).

\subsection*{Preserving convexity}
If $A$ and $B$ satisfy some property $P$, it is natural to ask when the subgroup they generate still enjoys $P$. The feature we focus on in the second half of the paper is \emph{hierarchical quasiconvexity} (HQC for short), which is the analogue of quasiconvexity in the world of hierarchically hyperbolic groups. A hierarchically quasiconvex subgroup of a HHG enjoys numerous geometric properties regarding, for example, a coarse median structure and a quadratic isoperimetric function \cite{Bow13, Bow18}, its asymptotic dimension \cite{asdim}, and the arrangement patterns of top-dimensional quasiflats \cite{quasiflat}. Therefore, it is relevant to understand when the subgroup generated by two hierarchically quasiconvex subgroups is again hierarchically quasiconvex. We provide sufficient conditions for our amalgamation procedure to preserve hierarchical quasiconvexity (the exact statement is Theorem~\ref{thm:amalgamation_of_quasiconvex}):
\begin{thmintro}\label{thmintro:amalgamation_HQC}
    Let $(G,\frakS)$ be a HHG, let $A,B\le G$ be two hierarchically quasiconvex subgroups, and let $C=A\cap B$. Suppose that:
    \begin{itemize}
        \item $A$ and $B$ satisfy the hypotheses of Theorem~\ref{thmintro:amalgamation}, for some $M\ge 0$; 
        \item $A$ and $B$ \emph{fill all squares} (Definition~\ref{defn:fill_all_gaps});
        \item $A$ and $B$ \emph{have no drift in the orthogonals} (Definition~\ref{defn:no_drift}).
    \end{itemize}
    There exists a constant $M_0\ge 0$, depending on $(G,\frakS)$ and the above data, such that, if $M\ge M_0$, then $\langle A,B\rangle_G\cong A*_C B$ is hierarchically quasiconvex in $(G,\frakS)$.
\end{thmintro}

Roughly, two HQC subgroups $A$ and $B$ fill all squares if, whenever two domains $U,V\in \frakS$ are orthogonal, if $A$ has large projection to $\C U$ and $B$ to $\C V$ then the intersection $A\cap B$ also has large projection to one of the domains. This property is equivalent to the fact that $A\cup B$ is hierarchically quasiconvex (see Lemma~\ref{lem:AUB}). 

Moreover, $A$ and $B$ have no drift in the orthogonals if it does \emph{not} happen that both $A$ and $B$ have bounded projections to some domain $U$, which is orthogonal to the domains $Y_a$ and $Y_b$ used to detect the amalgamation. In Subsection~\ref{subsec:no_drift} we provide a counterexample where the lack of this property falsifies the conclusion of Theorem~\ref{thmintro:amalgamation_HQC}.

\par\medskip
In our third and last Theorem, we study when our amalgamation procedure preserves \emph{strong quasiconvexity}. Recall that, given a metric space $X$, a subspace $Y\subseteq X$ is \emph{strongly quasiconvex} if every quasigeodesic $\gamma$ with endpoints on $Y$ lies in a neighbourhood of $Y$, whose radius only depends on $X$ and on the quasigeodesic constants of $\gamma$. Such a subset is also called \emph{Morse}, as strongly quasiconvex geodesics are exactly the Morse directions. Most of the properties of quasiconvex subsets of hyperbolic spaces hold for strongly quasiconvex subsets of general metric spaces \cite{Tran}. Furthermore, as explored in \cite{RST}, a subspace of a HHG is strongly quasiconvex if and only if it is hierarchically quasiconvex and enjoys a further assumption, the \emph{orthogonal projection dichotomy}. In Theorem~\ref{thm:amalgamation_of_strongly_quasiconvex_final_version} we prove that the latter property is preserved by our amalgamation procedure. This way, we provide a possible answer to \cite[Question 1]{RST}:
\begin{thmintro}\label{thmintro:amalgamation_strongqc}
Let $(G,\frakS)$ be a HHG, let $A,B\le G$ be two strongly quasiconvex subgroups of $G$, and let $C=A\cap B$. Suppose that $A$ and $B$ satisfy the hypotheses of Theorem~\ref{thm:amalgamation}, for some constant $M\ge 0$.
 
There exists a constant $M_0\ge 0$, depending on $(G,\frakS)$ and the strong quasiconvexity gauge of $A$ and $B$, such that if $M\ge M_0$ then $\langle A,B\rangle_G\cong A*_C B$ is strongly quasiconvex in $G$.
\end{thmintro}

In \cite[Theorem G]{RST_local_to_global}, Russell, Spriano, and Tran work in the context of \emph{local-to-global} groups, which include hierarchically hyperbolic groups by combining \cite[Theorem 4.18]{RST_local_to_global} and results from \cite{ABD}, and prove a combination theorem for \emph{stable} subgroups, i.e. strongly quasiconvex subgroups which are also \emph{hyperbolic}. Our Theorem~\ref{thmintro:amalgamation_strongqc} can therefore be seen as a generalisation of their result in the setting of hierarchically hyperbolic spaces, though the requirements of Theorem~\ref{thm:amalgamation} are stronger and specific to the hierarchical framework.

\subsection*{Towards relative hierarchical quasiconvexity} 
In a relatively hyperbolic group, the ``right'' notion of convexity of a subgroup is \emph{relative quasiconvexity}. Indeed, a relatively quasiconvex subgroup inherits a relative hyperbolic structure, and the intersection of two relatively quasiconvex subgroups is again relatively quasiconvex. As both HHG and relatively hyperbolic groups fall into the category of ``relative HHG'', in the sense of Definition~\ref{def:rel_HHG} below, one could look for a notion that unifies relative quasiconvexity and hierarchical quasiconvexity.

\begin{problintro}
    Let $(G,\frakS)$ be a relative HHG, and let $\frakS_0$ be a collection of domains which is closed under nesting and contains every $U\in\frakS$ such that $\C U$ is not hyperbolic. Formulate a notion of \emph{hierarchical quasiconvexity relative to $\frakS_0$} for subgroups of $G$, such that:
    \begin{itemize}
        \item If $(G,\frakS)$ is a HHG and $\frakS_0=\emptyset$, one recovers hierarchical quasiconvexity;
        \item If $(G,\mathbb P)$ is relatively hyperbolic and $\frakS_0=\mathbb P$, one recovers relative quasiconvexity;
        \item Under suitable conditions on $\frakS_0$, if a subgroup is HQC relative to $\frakS_0$ then it admits a structure of a \emph{relative hierarchically hyperbolic space} (see Definition~\ref{defn:relHHS});
        \item The intersection of two HQC subgroups relative to $\frakS_0$ is HQC relative to $\frakS_0$.
    \end{itemize}
\end{problintro}

We believe a possible approach would be to generalise the notion of \emph{transition points} on a geodesic in a relatively hyperbolic group, and then try to emulate the characterisation of relative quasiconvexity from \cite[Corollary 8.16]{hruska:qc}.

After one finds the right definition, one could attempt to extend Theorem~\ref{thmintro:amalgamation_HQC} to relative HQC subgroups, possibly generalising known combination theorems for relatively quasiconvex subgroups (see, among others, \cite{Gitik_quasiconvex_in_hyp, combination_quasiconv1, combination_quasiconv2}).

\subsection*{Outline}
Section~\ref{sec:HHS} provides the background on hierarchical hyperbolicity. In Section~\ref{sec:amalgamation} we prove the main amalgamation result, Theorem~\ref{thmintro:amalgamation}, which we then apply to certain geometric subgroups of mapping class groups in Section~\ref{sec:mcg}.

In Section~\ref{sec:combination_HQC} we strengthen our result to preserve hierarchical quasiconvexity. The proof of Theorem~\ref{thmintro:amalgamation_HQC} is in three steps. In Subsection~\ref{subsec:hqc}, we first recall that being HQC coincides with being almost closed under certain quasigeodesics, called \emph{hierarchy paths}. Next, in Subsection~\ref{subsec:aub} we determine under which conditions the union $A\cup B$ of two HQC subgroups is again HQC (see Lemma~\ref{lem:AUB}). Finally, given a hierarchy path connecting two points of $A\cup B$, in Subsection~\ref{subsec:amalgam_hqc} we show that it can be decomposed as a union of hierarchy paths with endpoints on cosets of $A\cup B$, and the conclusion follows from the hierarchical quasiconvexity of the latter cosets.

Finally, Section~\ref{sec:combination_strong} is devoted to the proof of the combination result for strongly quasiconvex subgroups, Theorem~\ref{thmintro:amalgamation_strongqc}.

\subsection*{Acknowledgements}
Firstly, I would like to thank my supervisor, Alessandro Sisto, for his constant support (even during the summer break) and several hints. Moreover, this paper arose as a side quest in a bigger, slightly unrelated project (talk about serendipity); so I am grateful to Yago Antolin, Matt Cordes, Giovanni Sartori, and Alessandro Sisto for numerous contributions to the first half of the paper, and for letting me publish it in this form. Finally, I thank the referee for the numerous comments that improved the exposition.

\section{A crash course in hierarchical hyperbolicity}\label{sec:HHS}
We start by recalling some notions from the world of \emph{hierarchically hyperbolic spaces and groups}, first introduced by Behrstock, Hagen, and Sisto in \cite{HHS_I}. 
\begin{defn}[Hierarchical space]\label{defn:Hierarchical_space}
The quasigeodesic space  $(X,\dist_{X})$ is a \emph{hierarchical space} if there exists $E\geq0$, called the \emph{hierarchical constant}, an index set $\frakS$, whose elements will be referred to as \emph{domains}, and a set $\{\C  U\mid U\in\frakS\}$ of geodesic metric spaces $(\C  U,\dist_U)$, called \emph{coordinate spaces},  such that the following conditions are satisfied:
\begin{enumerate}
\item\textbf{(Projections.)}\label{item:dfs_curve_complexes}
There is a set $\{\pi_U:  X\rightarrow 2^{\C  U}\mid U\in\frakS\}$ of \emph{projections} mapping points in $ X$ to sets of diameter bounded by $E$ in the various $\C  U\in\frakS$. Moreover, for all $U\in\frakS$, the coarse map $\pi_U$ is $(E,E)$--coarsely Lipschitz and $\pi_U( X)$ is $E$--quasiconvex in $\C  U$.

\item \textbf{(Nesting.)} \label{item:dfs_nesting}
$\frakS$ is equipped with a partial order $\nest$, and either $\frakS=\emptyset$ or $\frakS$ contains a unique $\nest$--maximal element, denoted by $S$. When $V\nest U$, we say $V$ is \emph{nested} in $U$. For each $U\in\frakS$, we denote by $\frakS_U$ the set of $V\in\frakS$ such that $V\nest U$. Moreover, for all $U,V\in\frakS$ with $V\propnest U$ there is a specified subset $\rho^V_U\subset\C  U$ with $\diam_{\C  U}(\rho^V_U)\leq E$. There is also a \emph{projection} $\rho^U_V: \C U\rightarrow 2^{\C V}$. (The similarity in notation is justified by viewing $\rho^V_U$ as a coarsely constant map $\C V\rightarrow 2^{\C  U}$.)
 
\item \textbf{(Orthogonality.)} \label{item:dfs_orthogonal}
$\frakS$ has a symmetric and anti-reflexive relation called \emph{orthogonality}: we write $U\orth V$ when $U,V$ are orthogonal. Also, whenever $V\nest U$ and $U\orth W$, we require that $V\orth W$. We require that for each $T\in\frakS$ and each $U\in\frakS_T$ such that $\{V\in\frakS_T\mid V\orth U\}\neq\emptyset$, there exists a domain $W\in\frakS_T-\{T\}$, which we call a \emph{container} for $U$ inside $T$, such that whenever $V\orth U$ and $V\nest T$, we have $V\nest W$. Finally, if $U \orth V$, then $U,V$ are not $\nest$--comparable.

\item \textbf{(Transversality and consistency.)}\label{item:dfs_transversal}
If $U,V\in\frakS$ are not orthogonal and neither is nested in the other, then we say $U,V$ are \emph{transverse}, denoted $U\transverse V$. In this case, there are sets $\rho^V_U\subseteq\C U$ and $\rho^U_V\subseteq\C  V$, each of diameter at most $E$ and satisfying the \emph{Behrstock inequality}:
$$\min\left\{\dist_{U}(\pi_U(z),\rho^V_U),\dist_{V}(\pi_V(z),\rho^U_V)\right\}\leq E$$
for all $z\in  X$.

For $U,V\in\frakS$ satisfying $V\nest U$ and for all $z\in X$, we have: 
$$\min\left\{\dist_{U}(\pi_U(z),\rho^V_U),\diam_{\C V}(\pi_V(z)\cup\rho^U_V(\pi_U(z)))\right\}\leq E.$$ 
 
The preceding two inequalities are the \emph{consistency inequalities} for points in $ X$.
 
Finally, if $U\nest V$, then $\dist_W(\rho^U_W,\rho^V_W)\leq E$ whenever $W\in\frakS$ satisfies either $V\propnest W$ or $V\transverse W$ and $W\not\bot U$.
 
\item \textbf{(Finite complexity.)} \label{item:dfs_complexity}
There exists $n\geq0$, the \emph{complexity} of $ X$ (with respect to $\frakS$), so that any set of pairwise--$\nest$--comparable elements has cardinality at most $n$.
  
\item \textbf{(Large links.)} \label{item:dfs_large_link_lemma}
Let $U\in\frakS$, let $z,z'\in X$ and let $N=\dist_{_U}(\pi_U(z),\pi_U(z'))$. Then there exists $\{T_i\}_{i=1,\dots,\lfloor N\rfloor}\subseteq\frakS_U- \{U\}$ such that, for any domain $T\in\mathfrak S_U-\{U\}$, either $T\in\frakS_{T_i}$ for some $i$, or $\dist_{T}(\pi_T(z),\pi_T(z'))<E$.  Also, $\dist_{U}(\pi_U(z),\rho^{T_i}_U)\leq N$ for each $i$.

\item \textbf{(Bounded geodesic image.)}\label{item:dfs:bounded_geodesic_image}
For all $U\in\frakS$, all $V\in\frakS_U- \{U\}$, and all geodesics $\gamma$ of $\C  U$, either $\diam_{\C  V}(\rho^U_V(\gamma))\leq E$ or $\gamma\cap N_E(\rho^V_U)\neq\emptyset$.
 
\item \textbf{(Partial realisation.)} \label{item:dfs_partial_realisation}
Let $\{V_j\}$ be a family of pairwise orthogonal elements of $\frakS$, and let $p_j\in \pi_{V_j}( X)\subseteq \C  V_j$. Then there exists $z\in  X$, which we call a \emph{partial realisation point} for the family, so that:
\begin{itemize}
\item $\dist_{V_j}(z,p_j)\leq E$ for all $j$,
\item for each $j$ and 
each $V\in\frakS$ with $V_j\nest V$, we have 
$\dist_{V}(z,\rho^{V_j}_V)\leq E$, and
\item for each $j$ and 
each $V\in\frakS$ with $V_j\transverse V$, we have $\dist_V(z,\rho^{V_j}_V)\leq E$.
\end{itemize}

\item\textbf{(Uniqueness.)} For each $\kappa\geq 0$, there exists
$\theta_u=\theta_u(\kappa)$ such that if $x,y\in X$ and
$\dist_{ X}(x,y)\geq\theta_u$, then there exists $V\in\frakS$ such
that $\dist_V(x,y)\geq \kappa$.\label{item:dfs_uniqueness}
\end{enumerate}
We often refer to $\frakS$, together with the nesting and orthogonality relations, and the projections as a \emph{hierarchical structure} for the space $X$. 
\end{defn}

\begin{rem}
    Notice that, if $E$ is a hierarchical constant for $(X, \frakS)$, then so is any $E'\ge E$. Hence, throughout the paper we will always implicitly assume that every hierarchical constant is strictly positive.
\end{rem}

\begin{rem}
Where it will not cause confusion, given $U\in\frakS$, we will often suppress the projection map $\pi_U$ when writing distances in $\C U$, i.e., given $A,B\subseteq X$ and $P\subseteq \C U$ we shall write $d_U(A,B)$ for $d_{\C U}(\pi_U(A),\pi_U(B))$ and  $d_U(A,P)$ for $d_{\C U}(\pi_U(A),P)$. Furthermore, when $V_1, V_2\in\frakS$ are such that $V_i\propnest U$ or $V_i\transverse U$ for $i=1,2$, we will write $d_U(V_1,V_2)$ for $d_{\C U}(\rho^{V_1}_U,\rho^{V_2}_U)$.
\end{rem}

\begin{defn}[(Relative) HHS]\label{defn:relHHS}
A hierarchical space is
\begin{itemize}
\item a \emph{hierarchically hyperbolic space} (HHS) if every coordinate space is $E$-hyperbolic;
\item a \emph{relatively hierarchically hyperbolic space} (relative HHS) if every coordinate space is either $E$-hyperbolic or $\nest$-minimal. Notice that this class includes relatively hyperbolic spaces and groups, as explained in \cite[Remark 1.15]{HHS_II}.
\end{itemize}
\end{defn}

All properties of hierarchically hyperbolic spaces whose proofs do not involve the hyperbolicity of coordinate spaces also hold for hierarchical spaces in general. In particular, the following is proved as in \cite[Lemma 1.5]{DHS}, which only uses the partial realisation axiom:

\begin{lem}\label{lem:close_proj_of_orthogonals} Let $(X, \frakS)$ be a hierarchical space. For every $U,V,W\in \frakS$ such that $U\orth V$ and both $\rho^U_W$, $\rho^V_W$ are defined, then $\dist_W(\rho^U_W,\rho^V_W)\le 2E$.
\end{lem}

Combining Lemma~\ref{lem:close_proj_of_orthogonals} and Axiom~\ref{item:dfs_transversal} we get:
\begin{cor}\label{cor:necessariamente_transverse}
    Let $U,V,W\in\frakS$. Suppose that $\dist_V(U,W)$ is well-defined and strictly greater than $2E$. Then $U\transverse W$.
\end{cor}

Moreover, possibly after enlarging the hierarchical constant $E$, we get the following variant of the bounded geodesic image axiom, which is proved by combining the consistency inequalities with the original bounded geodesic image axiom, as in \cite[Proposition 1.13]{HHS_II}:

\begin{lem}[BGI variant]\label{lem:bgi}
Let $(X, \frakS)$ be a hierarchical space. Let $x,y\in X$ and let $U,V\in\frakS$ be such that $U\propnest V$. Then either $\dist_U(x,y)\le E$ or every geodesic $[\pi_V(x), \pi_V(y)]\subseteq \C V$ must pass $E$-close to $\rho^U_V$.
\end{lem}


\begin{defn}[Consistent tuple]\label{defn:consistent_tuple}
Let $R\geq0$ and let $(b_U)_{U\in\frakS}\in\prod_{U\in\frakS}2^{\C   U}$ be a tuple such that for each $U\in\frakS$, the $U$--coordinate  $b_U$ has diameter $\leq R$.  Then $(b_U)_{U\in\frakS}$ is \emph{$R$--consistent} if for all $V,W\in\frakS$, we have $$\min\{\dist_V(b_V,\rho^W_V),\dist_W(b_W,\rho^V_W)\}\leq R$$
whenever $V\transverse W$ and 
$$\min\{\dist_W(b_W,\rho^V_W),\diam_V(b_V\cup\rho^W_V(b_W))\}\leq R$$
whenever $V\propnest W$.
\end{defn}

Later we will need the following, which is \cite[Theorem~3.1]{HHS_II} (notice that it holds for hierarchically \emph{hyperbolic} spaces):

\begin{thm}[Realisation]\label{thm:realisation}
Let $(X,\frakS)$ be a hierarchically hyperbolic space. Then for each $R\geq1$, there exists $\theta=\theta(R)$ so that,  for any $R$--consistent tuple $(b_U)_{U\in\frakS}$, there exists $x\in X$ such that $\dist_V(x,b_V)\leq\theta$ for all $V\in\frakS$.
\end{thm}

\noindent Observe that the uniqueness axiom (Definition~\eqref{item:dfs_uniqueness}) implies that the \emph{realisation point} $x$ for $(b_U)_{U\in\frakS}$, provided by Theorem~\ref{thm:realisation}, is coarsely unique.

\begin{defn}[Automorphism]\label{defn:auto}
Let $({X},\frakS)$ be a hierarchical space. An \emph{automorphism} consists of a map $g:\,{X}\to {X}$, a bijection $g^\sharp:\, \frakS\to \frakS$ preserving nesting and orthogonality, and, for each $U\in\frakS$, an isometry $g^\diamond(U):\,\C  U\to \C (g^\sharp(U))$ for which the following two diagrams commute for all $U,V\in\frakS$ such that $U\propnest V$ or $U\transverse V$:
$$\begin{tikzcd}
{X}\ar{r}{g}\ar{d}{\pi_U}&{X}\ar{d}{\pi_{g^\sharp (U)}}\\
\C  U\ar{r}{g^\diamond (U)}&\C  (g^\sharp (U))\\
\end{tikzcd}$$
and
$$\begin{tikzcd}
\C  U\ar{r}{g^\diamond (U)}\ar{d}{\rho^U_V}&\C  (g^\sharp (U))\ar{d}{\rho^{g^\sharp (U)}_{g^\sharp (V)}}\\
\C  V\ar{r}{g^\diamond (V)}&\C  (g^\sharp (V))\\
\end{tikzcd}$$
Whenever it will not cause ambiguity, we will abuse notation by dropping the superscripts and just calling all maps $g$.
\end{defn}

We say that two automorphisms $g,g'$ are \emph{equivalent}, and we write $g\sim g'$, if $g^\sharp=(g')^\sharp$ and $g^\diamond(U)=(g')^\diamond(U)$ for each $U\in\frakS$.  Given an automorphism $g$, a quasi-inverse $\ov{g}$ for $g$ is an automorphism with $\ov{g}^\sharp=(g^\sharp)^{-1}$ and such that, for every $U\in \frakS$, $\ov{g}^\diamond(U)=g^\diamond(U)^{-1}$. Since the composition of two automorphisms is an automorphism, the set of equivalence classes of automorphisms forms a group, denoted $\text{Aut}(\frakS)$.

\begin{defn}\label{defn:action_on_hhs}
    A finitely generated group $G$ \emph{acts} on a hierarchical space $({X},\frakS)$ by automorphisms if there is a group homomorphism $G\to \text{Aut}(\frakS)$. Notice that this induces a $G$-action on $X$ by uniform quasi-isometries.
\end{defn}

\begin{defn}[(Relative) HHG]\label{def:rel_HHG}
If a group $G$ acts on a (relative) HHS $(X,\frakS)$, in such a way that the action on $X$ is metrically proper and cobounded and the action on $\frakS$ is cofinite, then $G$ is called a \emph{(relative) hierarchically hyperbolic group}, or (relative) HHG for short. Any quasi-isometry between $G$ and $X$ given by the Milnor-\v{S}varc Lemma endows $G$ with the (relative) HHS structure of $X$ (possibly with a larger hierarchical constant). 
\end{defn}


\section{Detecting an amalgamated free product in a HHG}\label{sec:amalgamation}
We are now ready to prove Theorem~\ref{thmintro:amalgamation} from the introduction, in the following extended form.
\begin{thm}\label{thm:amalgamation}
    Let $G$ be a group acting on a hierarchical space $(X,\frakS)$, and fix a basepoint $x_0\in X$.  Let $A_1,\ldots,A_n\le G$ be subgroups, and let $C$ be a subset of the intersection $\bigcap_{l=1}^n A_l$. Suppose that there exists a constant $M\ge 100E$, where $E$ is a hierarchical constant for $X$, and a domain $Y_a\in \frakS$ for every $a\in (\bigcup_i A_i)-C$, such that the following hold:
    \begin{enumerate}[label=(\Roman*)]
        \item \label{hyp::1} $\max\left\{\diam_{Y_a}(Cx_0), \diam_{Y_a}(aCx_0)\right\}\le M/10$;
        \item \label{hyp::4} $\dist_{Y_a}(Cx_0, aCx_0)\ge M$; 
        \item \label{hyp::2} If $a\in A_i-C$ and $b\in A_j-C$ where $i\neq j$, $Y_a\transverse aY_{b}$;
        \item \label{hyp::3} In the same setting, $\dist_{Y_a}(Cx_0, bCx_0)\le M/10$.
    \end{enumerate}
    Then $C$ coincides with the pairwise intersections $A_i\cap A_j$ for every $i\neq j$, and in particular $C = \bigcap_{l=1}^n A_l$. Moreover, the natural map $$\Asterisk_{C} A_\bullet:= A_1 *_C \ldots *_C A_n\to \langle A_1,\ldots,A_n\rangle_G$$
    is an isomorphism. 
\end{thm}

\begin{rem}[Outline] Given a word $w=g_1\ldots g_k c\in \Asterisk_{C} A_\bullet-\{1\}$, we shall show that one can find a collection of pairwise transverse domains $W_1,\ldots, W_k$, one for every subword of $w$ of the form $g_1\ldots g_j$, such that the projection of $g_1\ldots g_j Cx_0$ on $W_i$ only changes when one passes from $g_1\ldots g_{i-1} Cx_0$ to $g_1\ldots g_{i} Cx_0$. In a way, the domain $W_i$ detects “the $i$-th step of $w$”, and as a consequence of assumption~\ref{hyp::3} all subwords $g_1\ldots g_j$ with $j\ge i+1$ cannot undo the translation on $W_i$. In particular, one gets that $\dist_{W_k}(Cx_0, wCx_0)$ is greater than some positive constant, depending on $M$ and $E$, and therefore $w$ is non-trivial in $G$ as it acts non-trivially on $(X,\frakS)$. 
\end{rem}

We now move to the proof of Theorem~\ref{thm:amalgamation}. First notice that, whenever $i\neq j$, $C=A_i\cap A_j$. Indeed $C\subseteq A_i\cap A_j$, and if $a\in (A_i\cap A_j)-C$ then by assumption~\ref{hyp::3} applied to $a=b$ we would get that $\dist_{Y_a}(Cx_0, aCx_0)\le M/10$, in contrast with assumption~\ref{hyp::4}.
\par\medskip

Now, we want to prove that the natural epimorphism $\Asterisk_{C} A_\bullet\to \langle A_1,\ldots,A_n\rangle_G$ is injective. In other words, given any non-trivial word $w= g_1\ldots g_k c\in \Asterisk_{C} A_\bullet-\{1\}$, where  $c\in C$ and every two consecutive $g_i$ and $g_{i+1}$ belong to different factors, we have to show that $w\neq_G 1$. This is clearly true if $k\le 2$, so we focus on the case $k\ge3$. 

\begin{notation}\label{notation:C_iW_i}
    For every $i=1,\ldots, k$ let $Y_i=Y_{g_i}$, and set 
    $$C_i=g_1\ldots g_{i-1}Cx_0,\quad W_i=g_1\ldots g_{i-1}Y_i,$$ so that
    $$\dist_{W_i}(C_i,C_{i+1})= \dist_{Y_i}(Cx_0,g_iCx_0)\ge M.$$
\end{notation}
    We break the rest of the proof of Theorem~\ref{thm:amalgamation} into a series of Claims. 

    \begin{claim}\label{claim:transverse}
        $W_i\transverse W_{i+1}$ for every $i=1,\ldots, k-1$.
    \end{claim}
    \begin{proof}[Proof of Claim~\ref{claim:transverse}]
    Notice that $Y_i$ and $g_iY_{i+1}$ are transverse by assumption~\ref{hyp::2}, since $g_i$ and $g_{i+1}$ lie in different factors of the amalgamation. Therefore, the domains $W_i=(g_1\ldots g_{i-1})Y_i$ and $W_{i+1}=(g_1\ldots g_{i-1})g_i Y_{i+1}$ must be transverse as well, since the $G$-action preserves transversality. 
    \end{proof}

    \begin{claim}\label{claim:projection of adjacent Wi}
        For every $i=2, \ldots, k$, $\dist_{W_i}(W_{i-1}, C_i)\le M/10+E$. 
        
        Symmetrically, for every $i=1, \ldots, k-1$, $\dist_{W_i}(W_{i+1}, C_{i+1})\le M/10+E$.
    \end{claim}

    \begin{proof}
        We only prove the first statement, as the second follows analogously. As $\dist_{W_{i-1}}(C_{i-1}, C_i)\ge M\ge 4E$, by the Behrstock inequality one of the following happens:
        \begin{itemize}
            \item Every point of $\pi_{W_i}(C_i)$ is $E$-close to $\rho^{W_{i-1}}_{W_i}$. Then the conclusion clearly follows.
            \item Every point of $\pi_{W_i}(C_{i-1})$ is $E$-close to $\rho^{W_{i-1}}_{W_i}$. Then we have that
            $$\dist_{W_i}(C_i, W_{i-1})\le \dist_{W_i}(C_i, C_{i-1})+E\le M/10+E,$$
            where we used Assumption~\ref{hyp::3}. 
        \end{itemize}
    \end{proof}

    \begin{claim}\label{claim:distance_of_consecutives}
        For every $i=2, \ldots, k-1$, $\dist_{W_i}(W_{i-1}, W_{i+1})\ge 4/5M-4E>6E$.  
    \end{claim}
    \begin{proof}[Proof of Claim~\ref{claim:distance_of_consecutives}]
    We simply notice that
    $$\dist_{W_i}(W_{i-1}, W_{i+1})\ge \dist_{W_i}(C_{i}, C_{i+1})-\dist_{W_i}(C_{i}, W_{i-1})-\dist_{W_i}(C_{i+1}, W_{i+1})-$$
    $$-\diam_{W_i}(\rho^{W_{i-1}}_{W_i})-\diam_{W_i}(\rho^{W_{i+1}}_{W_i})\ge 4/5M-4E.$$
    \end{proof}



    We continue with a general statement about families of pairwise transverse domains in a hierarchical space:
    \begin{lem}[Transverse domains in a row]\label{lemma:alex_lemma}
        Let $(X,\frakS)$ be a hierarchical space, and let $\{W_1, \ldots, W_k\}\subset \frakS$ be a collection of domains such that  $W_i\transverse W_{i+1}$ for every $i=1,\ldots, k-1$. If $\dist_{W_j}(W_{j-1},W_{j+1})> 6E$ for every $j=2,\ldots, k-1$, then $\{W_1, \ldots, W_k\}$ are pairwise transverse, and $\dist_{W_j}(W_{i},W_{r})> 2E$ for every $i<j<r$.
    \end{lem}
    Notice that the hypothesis of Lemma~\ref{lemma:alex_lemma} are satisfied by the $W_i$s we are considering, in view of Claims~\ref{claim:transverse} and~\ref{claim:distance_of_consecutives} and of the fact that $M\ge 100E$.

    \begin{proof}[Proof of Lemma~\ref{lemma:alex_lemma}]
        We proceed by induction on $k$. If $k=3$ we just need to show that $W_1\transverse W_3$. This is true since $\dist_{W_2}(W_{1},W_{3})> 6E$, and we can invoke Corollary~\ref{cor:necessariamente_transverse}.

        Now assume that the theorem is true for every collection of at most $k-1$ elements. In particular, by applying the inductive hypothesis to the collections $\{W_1, \ldots, W_{k-1}\}$ and to $\{W_2, \ldots, W_{k}\}$, we get that $\dist_{W_j}(W_{i},W_{r})> 2E$ whenever $i<j<r$ and $(i,r)\neq(1,k)$. Thus, we only need to show that $\dist_{W_j}(W_{1},W_{k})> 2E$ for every $j=2,\ldots, k-1$, and again Corollary~\ref{cor:necessariamente_transverse} will imply that $W_1\transverse W_k$. Now
        $$\dist_{W_j}(W_{1},W_{k})\ge \dist_{W_j}(W_{j-1},W_{j+1})-\dist_{W_j}(W_{j-1},W_{1})-$$
        $$-\dist_{W_j}(W_{j+1},W_{k})-\diam_{W_j}(\rho^{W_{j-1}}_{W_j})-\diam_{W_j}(\rho^{W_{j+1}}_{W_j}).$$
        Since by Behrstock inequality we have that $\dist_{W_j}(W_{j-1},W_{1})\le E$, and similarly $\dist_{W_j}(W_{j+1},W_{k})\le E$, we get that 
        $$\dist_{W_j}(W_{1},W_{k})> 6E-4E= 2E,$$
        as required.
    \end{proof}

    \begin{claim}\label{claim:W1Wk}
    For every $i\neq j$ and every $x\in C_j$, $\dist_{W_i}(x, W_j)\le E$.
    \end{claim}

    \begin{proof}[Proof of Claim~\ref{claim:W1Wk}]
        We assume that $j<i$, as the case $j>i$ is dealt with analogously. If by contradiction $\dist_{W_i}(x, W_j)> E$, then the Behrstock inequality implies that $\dist_{W_j}(x, W_i)\le E$. Moreover $\dist_{W_j}(W_i,W_{j+1})\le E$, either by combining the Behrstock inequality with Lemma~\ref{lemma:alex_lemma} (if $j\le i-2$) or because $W_i=W_{j+1}$ (if $j=i-1$). Thus we get
         $$\dist_{W_j}(x,C_{j+1})\le \dist_{W_j}(x,W_i)+\diam_{W_j}(\rho^{W_i}_{W_j})+\dist_{W_j}(W_i,W_{j+1})+$$
         $$+\diam_{W_j}(\rho^{W_{i+1}}_{W_j})+\dist_{W_j}(W_{j+1},C_{j+1})\le M/10+5E, $$
        where we used that $\dist_{W_j}(W_{j+1},C_{j+1})\le M/10+5E$ by Claim~\ref{claim:projection of adjacent Wi}. But then
        $$M\le \dist_{W_j}(C_j,C_{j+1}) \le \dist_{W_j}(x,C_{j+1}) M/10+5E,$$
        giving a contradiction as $M\ge 100E$.
    \end{proof}

    \begin{claim}\label{claim:large projection}
        $\dist_{W_k}(C x_0, w Cx_0)=\dist_{W_k}(C_1, C_{k+1})\ge 9/10M-5E>0.$
    \end{claim}
    The Claim concludes the proof of Theorem~\ref{thm:amalgamation}, because it proves that $w$ acts non-trivially on $X$, and therefore is non-trivial in $G$.
    
    \begin{proof}[Proof of Claim~\ref{claim:large projection}]
    This is just a matter of putting all the above Claims together. Indeed, one has that
     $$\dist_{W_k}(C_1, C_{k+1})\ge \dist_{W_k}(C_{k},C_{k+1})-\dist_{W_k}(C_k, W_{k-1})-\diam_{W_k}\left(\rho^{W_{k-1}}_{W_k}\cup\rho^{W_{1}}_{W_k}\cup C_1\right). $$
    The first term of the right-hand side is at least $M$ by assumption~\ref{hyp::4}. The second one is at most $M/10+E$ by Claim~\ref{claim:projection of adjacent Wi}. Regarding the third one, we have that 
    $$\diam_{W_k}\left(\rho^{W_{k-1}}_{W_k}\cup\rho^{W_{1}}_{W_k}\cup C_1\right)\le$$
    $$\le \diam_{W_k}\left(\rho^{W_{k-1}}_{W_k}\right)+\dist_{W_k}(W_{k-1},W_1)+\diam_{W_k}\left(\rho^{W_{1}}_{W_k}\cup C_1\right)\le E+E+(2E)=4E, $$
    where we used Claim~\ref{claim:W1Wk} to bound the last term. Hence $$\dist_{W_k}(C_1, C_{k+1})\le 9/10 M-5E,$$ as required.
    \end{proof}

Before moving forward, we point out the following lemma, which combines some of the above Claims and will be useful later.



    
\begin{lem}\label{lem:Cj_close_in_W_i} For every $1\le j<i\le k$ and every $x\in C_j$,
$$\dist_{W_i}(x, C_i)\le M/10+5E.$$
Symmetrically, for every $1\le i+1<j\le k$ and every $x\in C_j$,
$$\dist_{W_{i}}(x, C_{i+1})\le M/10+5E.$$
\end{lem}

\begin{proof} 
This is just a combination of some inequalities from the above proof. Indeed
$$\dist_{W_i}(x, C_i)\le\dist_{W_i}(x,W_j)+\diam_{W_i}(W_j)+\dist_{W_i}(W_j, W_{i-1})+\diam_{W_i}(W_{i-1})+\dist_{W_i}(W_{i-1}, C_i).$$
The first term is at most $E$ by Claim~\ref{claim:W1Wk}; the third term is at most $E$ by the Behrstock inequality, combined with Lemma~\ref{lemma:alex_lemma}; the last term is at most $M/10+E$ by Claim~\ref{claim:projection of adjacent Wi}. Thus
$$\dist_{W_i}(x, C_i)\le 4E+M/10+E=M/10+5E.$$
The second inequality follows analogously. 
\end{proof}

\section{Amalgamation of geometric subgroups along common boundaries}\label{sec:mcg}
We not describe an application of Theorem~\ref{thm:amalgamation} to mapping class groups of finite-type surfaces. The subgroups we shall amalgam are mapping class groups of embedded subsurfaces, which overlap “sufficiently” away from a collection of common boundary components. The precise result is Theorem~\ref{thm:geometric} below, which in turn is Theorem~\ref{thmintro:geometric} from the introduction. 

\subsection{Notation and setting}
We gather here the (fairly standard, but sometimes subtle) notation we shall need to prove Theorem~\ref{thm:geometric}.
\subsubsection{Curves, surfaces, and mapping classes}
In what follows, let $S$ be a possibly disconnected surface of finite-type (that is, an oriented, compact surface from which a finite number of points is removed). If $S$ is connected, let $\MCG(S)$ be the (extended) mapping class group of $S$, that is, the group of isotopy class of self-homeomorphisms of $S$ fixing the boundary pointwise. If instead $S$ is disconnected, its mapping class group is defined as the direct product of the mapping class groups of its connected components.

By \emph{curve} we denote the isotopy class of an embedding $\mathbb{S}^1\hookrightarrow S$ which is \emph{essential}, meaning that it does not bound a disk with at most one puncture. A \emph{multicurve} is a collection of pairwise disjoint, non-isotopic curves. We often see a multicurve as a subsurface of $S$, by replacing each curve with a closed annulus whose core is the curve. Given a curve $\gamma$, let $T_\gamma$ be the Dehn Twist around $\gamma$ (see e.g. \cite[Section 3.1.1]{FarbMargalit}).

The \emph{curve graph} $\C S$ is the simplicial graph whose vertices are all curves on $S$, and where adjacency corresponds to disjointness. This definition does not apply to some surfaces of small complexity:
\begin{itemize}
    \item If $S$ is either a sphere with four punctures, or a torus with at most two punctures, then two curves are adjacent in $\C S$ if and only if their intersection number is minimal among all pairs.
    \item If $S$ is an annulus then its \emph{annular curve graph}, which we still denote by $\C S$, is a quasiline. We won't need the actual definition, and we refer to \cite[Section 2.4]{MasurMinsky2} for further explanations.
\end{itemize}  

The following is an easy consequence of known facts, but we provide a proof for completeness. Recall that a subgroup $H$ of a group $G$ is \emph{separable} if, for every $g\in G-H$, there exists a finite quotient $\psi\colon G\to \overline G$ such that $\psi(g)\not \in \psi(H)$. 
\begin{lem}\label{lem:separation_boundary}
    Let $S$ be a finite-type surface, and let $\Gamma\subseteq \partial S$ be a collection of boundary components. Then the Dehn twist flat $\Z^{|\Gamma|}$ supported on $\Gamma$ is a separable subgroup of $\MCG(S)$. 
\end{lem}

\begin{proof}
Let $\widehat S$ be the surface obtained from $S$ by gluing a once-punctured disk  onto each boundary component belonging to $\Gamma$, and let $\{p_1,\ldots, p_r\}$ be the punctures added this way. By e.g. \cite[Proposition 3.19]{FarbMargalit}, the quotient $\MCG(S)/\Z^{|\Gamma|}$ is isomorphic to the group $\MCG(\widehat S, \{p_1,\ldots,p_k\})$, which is the subgroup of $\MCG(\widehat S)$ of all mapping classes fixing the punctures $\{p_1,\ldots,p_k\}$ pointwise. 

Now let $g\in \MCG(S)-\Z^{|\Gamma|}$, and let $h$ be its image in $\MCG(\widehat S, \{p_1,\ldots,p_k\})$, which is therefore non-trivial. Now, $\MCG(\widehat S, \{p_1,\ldots,p_k\})\le \MCG(\widehat S)$, and the latter is residually finite (see e.g. \cite[Theorem 6.11]{FarbMargalit} and the discussion below it for punctured surfaces); hence we can find a finite quotient $\ov G$ of $\MCG(\widehat S, \{p_1,\ldots,p_k\})$ where $h$ projects non-trivially. Then, if we take the composition $\MCG(S)\to \MCG(\widehat S, \{p_1,\ldots,p_k\})\to \ov G$, we get a finite quotient such that the image of $\Z^{|\Gamma|}$ is trivial while the image of $g$ is not.
\end{proof}

\subsubsection{The marking graph}
Let $\M (S)$ be the \emph{marking graph} of $S$, as defined in \cite[Section 2.5]{MasurMinsky2}. A vertex $x$ of $\M(S)$ consists of a multicurve of maximal cardinality, called the \emph{support} of $x$ and denoted by $\text{supp}(x)$, and for every $\alpha\in \text{supp}(x)$ a choice of a set $p\in \C \alpha$ of diameter at most $1$, called the \emph{transversal} associated to $\alpha$. By e.g. \cite[Theorem 11.1]{HHS_II}, which in turn builds on observations from \cite{MasurMinsky2}, $\M (S)$ has the following HHS structure:
    \begin{itemize}
        \item The domain set $\frakS$ is the collection of subsurfaces of $S$;
        \item $\nest$ is containment and $\orth$ is disjointness of subsurfaces (up to isotopy);
        \item For every $U\in \frakS$, the associated coordinate space $\C U$ is the curve graph of $U$, and the projection $\pi_U\colon \M (S)\to \C U$ is the subsurface projection.
    \end{itemize}
    Furthermore, $\MCG(S)$ acts geometrically on $\M(S)$, and therefore inherits the HHG structure described above.

\subsection{The result}
By \emph{subsurface} we mean the isotopy class of an embedding $U\hookrightarrow S$, where $U$ is a finite-type surface whose connected components cannot be pairs of pants. The \emph{boundary} $\partial U$ of a connected, non-annular subsurface $U$ is defined as the closure of $U$ minus its interior: if $U$ is an annulus, with a little abuse of notation we define its boundary as the core of the annulus. The boundary of a disconnected subsurface is the union of the boundaries of its connected components.

If $U$ is a \emph{closed} subsurface, we can extend every mapping class on $U$ to the identity on $S-U$, and we get a homomorphism $\MCG(U)\to \MCG(S)$. This map is injective if and only if every curve in $\partial U$ is essential in $S$, and no two boundary components are isotopic (see e.g. \cite[Theorem 3.18]{FarbMargalit}). In this case, we call $\MCG(U)$ a \emph{geometric embedded} subgroup of $\MCG(S)$. Our main theorem describes when two such subgroups span a free product, amalgamated along a common boundary Dehn twist flat:

\begin{thm}\label{thm:geometric}
   Let $S$ be a connected finite-type surface, and let $\MCG(U)$ and $\MCG(V)$ be two geometric embedded subgroups, where each of $U$ and $V$ is either connected or a multicurve.

   Let $\Gamma$ be the collection of all curves which belong to the boundary of both $U$ and $V$, and let $\Gamma_U=\partial U-\Gamma$ (resp. $\Gamma_V=\partial V-\Gamma$). Suppose that $\Gamma_U$, $\Gamma_V$ are both non-empty, and that $\dist_{\C (S-\Gamma)}(\Gamma_U,\Gamma_V)\ge 4$.

   Then there exist finite index subgroups $A\le \MCG(U)$, $B\le \MCG(V)$, intersecting along the Dehn twist flat $\Z^{|\Gamma|}$, such that
   $$\langle A,B\rangle_{\MCG(S)}\cong A *_{\Z^{|\Gamma|}} B.$$
\end{thm}

The proof will highlight some useful techniques to verify the requirements of Theorem~\ref{thm:amalgamation}. Such tools often use only the existence of a HHG structure for the ambient group, and can therefore be exported to more general settings. 

\begin{proof}
We start with some considerations on the two subsurfaces and their mapping class groups:
    \begin{claim}\label{claim:W_W'}
        Given two subsurfaces $W\nest U$ and $W'\nest V$, if none of the two subsurfaces is a sub-multicurve of $\Gamma$, then $W\transverse W'$.
    \end{claim} 
    \begin{proof}[Proof of claim \ref{claim:W_W'}]
If by contradiction $W$ and $W'$ were either disjoint, or one contained in the other, then there would be two curves $\delta\subset W$ and $\delta'\subset W$, both disjoint from $\Gamma$, such that $\dist_{\C (S-\Gamma)}(\delta, \delta')\le 1$. However, notice that $\delta$ would either belong to, or be disjoint from, $\partial U$, since $W\nest U$, and the same is true for $\delta'$ and $\partial V$. Therefore
        $\dist_{\C (S-\Gamma)}(\delta, \delta')\ge \dist_{\C (S-\Gamma)}(\Gamma_U, \Gamma_V) -2 \ge 2,$ contradicting our assumption.
    \end{proof}

    \begin{claim}\label{claim:Mu_cap_Mv}
        $\MCG(U)\cap\MCG(V)=\Z^{|\Gamma|}$.
    \end{claim}
    \begin{proof}[Proof of Claim~\ref{claim:Mu_cap_Mv}]
        Clearly $\Z^{|\Gamma|}\le \MCG(U)\cap \MCG(V)$. Conversely, pick an element $g\in \MCG(U)\cap \MCG(V)$. By e.g. \cite[Corollary 13.3]{FarbMargalit}, there exists a power $k$ such that $g^k=\prod_{i=1}^l g_i$, where each $g_i$ is either a partial pseudo-Anosov or a power of a Dehn Twist, and the supports $\{R_i\}$ of the $g_i$ are all pairwise disjoint, closed subsurfaces. 
        
        Now, as $g$ is supported on both $U$ and $V$, each $R_i$ must be nested in both $U$ and $V$, and therefore must be nested in $\Gamma$ by Claim~\ref{claim:W_W'}. In other words, $g^k\in \Z^{|\Gamma|}$.
        
        Now, if one between $U$ and $V$ is a multicurve we immediately get that $g\in \Z^{|\Gamma|}$ as well, and we are done. Otherwise, the above argument shows that $g$ projects to a torsion element of $\MCG(\widehat U)$, where $\widehat U$ is the surface obtained by capping every curve in $\Gamma$ with a once-punctured disk. However, as $\partial U-\Gamma$ is non-empty, $\widehat U$ still has non-empty boundary, so \cite[Corollary 7.3]{FarbMargalit} yields that $\MCG(\widehat U)$ is torsion-free (more precisely, the cited result is stated for surfaces with negative Euler characteristic; however we can always embed $\MCG(\widehat U)$ in the mapping class group of such a surface, for example by gluing a sphere with five disks removed along one of the boundaries of $\widehat U$). Then again $g\in \Z^{|\Gamma|}$, as required.
    \end{proof}

Next, we produce the subgroups $A, B$. Extend $\Gamma$ to a multicurve $\Gamma'$ of maximal cardinality, and let $x_0\in \M(S)$ be any marking supported on $\Gamma'$. Let
    \begin{equation}\label{eq:C0}
        C_0\coloneq \sup\{\dist_W(x_0,V)\,|\,W\nest U,W\not\nest \Gamma\}+\sup\{\dist_{W'}(x_0,U)\,|\,W'\nest V,W'\not\nest \Gamma\}.
    \end{equation}
The above quantity is well-defined, as Claim~\ref{claim:W_W'} tells us that, whenever $W\nest U$ is not a sub-multicurve of $\Gamma$, we have that $W\transverse V$, so that the projection of $V$ to $W$ is well-defined. Furthermore, $C_0$ is finite. Indeed, let $y_V\in \M (S)$ be a marking whose support contains $\partial V$. Then, whenever $W$ is nested in $U$ but not in $\Gamma$, the projection $\rho^V_W$ coincides with the subsurface projection of $y_V$ to $W$. In other words, 
    $$\dist_W(x_0,V)=\dist_W(x_0,y_V)\le E\dist_{\M (S)}(x_0,y_V)+E,$$
    where $E$ is a hierarchical constant for $(\M (S), \frakS)$, and we used that subsurface projections are $E$-coarsely Lipschitz. We can similarly define a marking $y_U$ whose support contains $\partial U$, and use it to bound the second term of Equation~\eqref{eq:C0}. Notice that, in view of the above argument, we can rewrite Equation~\eqref{eq:C0} in the following form:
    \begin{equation}\label{eq:C0_con_y}
        C_0= \sup\{\dist_W(x_0,y_V)\,|\,W\nest U,W\not\nest \Gamma\}+\sup\{\dist_{W'}(x_0,y_U)\,|\,W'\nest V,W'\not\nest \Gamma\}.
    \end{equation}

    Now, let $F\subset \MCG(U)$ be the subset of all elements $a$ such that, for every $Y\in\frakS$, $$\dist_Y(y_U, a(y_U))\le 12C_0+100E+100$$
    Notice that $F$ is finite, by the fact that $\MCG(S)$ acts metrically properly on $\M(S)$, combined with the uniqueness axiom~\eqref{item:dfs_uniqueness} for the HHS $(\M(S),\frakS)$. By Lemma~\ref{lem:separation_boundary} there exists a finite-index subgroup $A\le \MCG(U)$ containing $\Z^{|\Gamma|}$, such that if $a\in F\cap A$ then $a\in \Z^{|\Gamma|}$. One can define $B$ analogously. Notice that $$\Z^{|\Gamma|}\le A\cap B\le \MCG(U)\cap \MCG(V)=\Z^{|\Gamma|},$$
    where the last equality is Claim~\ref{claim:Mu_cap_Mv}.

We are finally ready to show that $A$ and $B$ satisfy the hypotheses of Theorem~\ref{thm:amalgamation}, and therefore
$$\langle A,B\rangle_{\MCG(S)}\cong A *_{\Z^{|\Gamma|}} B.$$
We first describe the subsurfaces $Y_g$ we shall use to prove the Theorem. For every $a\in A-\Z^{|\Gamma|}$ there exists a domain $Y_a$ such that $\dist_{Y_a}(y_U, a(y_U))\ge 12C_0+100E+100$, by how we chose $A$.

We first notice that $Y_a$ must be nested in $U$. Indeed, the action of $\MCG(U)$ fixes $\partial U$ and every curve $\alpha\in \text{supp}(y_U)$ which is disjoint from $U$, and the transversal for each such $\alpha$ is moved within distance at most $4$ in $\C\alpha$ (see e.g. \cite{Mousley}). Furthermore, if a surface $Y$ is either disjoint from $U$, or properly contains $U$, then the subsurface projection of both $y_U$ and $a(y_U)$ only depends on the above data. Therefore $\dist_Y(y_U, a(y_U))\le 4$ whenever $Y\not\nest U$. 

Furthermore, suppose that every $Y_a$ as above was nested in $\Gamma$. Then we could multiply $a$ by a suitable multitwist in $\Z^{|\Gamma|}$ to find an element of $(A\cap F)-\Z^{|\Gamma|}$, contradicting our choice of $A$. Thus pick any $Y_a$ as above which is nested in $U$ but not in $\Gamma$. 

Now let $M=10C_0+100E+40$, and choose $x_0$ as above. We now verify the four requirements from Theorem~\ref{thm:amalgamation}, for every $a\in A-\Z^{|\Gamma|}$ and every $b\in B-\Z^{|\Gamma|}$:
\begin{itemize}
\item \ref{hyp::1} The action of $\Z^{|\Gamma|}$ fixes every curve $\alpha\in \Gamma'-\Gamma$, and the associated transversal (again, up to distance $4$ in $\C\alpha)$. Therefore, as $Y_a\nest U$ but is not nested in $\Gamma$, the subsurface projection of $\Z^{|\Gamma|} x_0$ to $Y_a$ is coarsely the same as that of $x_0$, and the same is true for $a\Z^{|\Gamma|} x_0$. In particular $\max\{\diam_{Y_a}(\Z^{|\Gamma|} x_0), \diam_{Y_a}(a\Z^{|\Gamma|} x_0)\}\le 4\le M/10$.
\item \ref{hyp::4} By the above discussion we get $\dist_{Y_a}(\Z^{|\Gamma|} x_0,a\Z^{|\Gamma|} x_0)\ge\dist_{Y_a}(x_0,ax_0)-8$. Furthermore  Equation~\eqref{eq:C0_con_y} yields that
$$\dist_{Y_a}(x_0,ax_0)-8\ge \dist_{Y_a}(y_U,ay_U)-2C_0-8\ge M.$$
\item \ref{hyp::2} As $Y_a\nest U$, $Y_b\nest V$ and none is nested in $\Gamma$, Claim~\ref{claim:W_W'} tells us that $Y_a\transverse Y_b$.
\item \ref{hyp::3} Firstly $\dist_{Y_b}(\Z^{|\Gamma|} x_0,a\Z^{|\Gamma|} x_0)\le \dist_{Y_b}(x_0,ax_0)$. Moreover
$$\dist_{Y_a}(ax_0, Y_b)\ge \dist_{Y_a}(ax_0, x_0)-\dist_{Y_a}(x_0, Y_b)-\diam_{Y_a}(\rho^{Y_b}_{Y_a})\ge M-C_0-E\ge 2E. $$
Hence by the Behrstock inequality $\dist_{Y_b}(ax_0, Y_a)\le E$. In turn, this means that
$$\dist_{Y_b}(x_0, ax_0)\le \dist_{Y_b}(x_0, Y_a)+ \diam_{Y_b}(\rho^{Y_a}_{Y_b})+\dist_{Y_b}(Y_a, ax_0)\le C_0+2E\le M/10.$$
\end{itemize}
The proof of Theorem~\ref{thm:geometric} is now complete.\end{proof}

\section{A combination theorem for hierarchically quasiconvex subgroups}\label{sec:combination_HQC}
This Section is devoted to the proof of Theorem~\ref{thmintro:amalgamation_HQC} from the introduction, regarding the hierarchical quasiconvexity of the amalgamated free product of two HQC subgroups $A,B$ of a HHG $(G,\frakS)$.

\subsection{Quasiconvexity and friends}\label{subsec:hqc}
\begin{defn}\label{defn:quasiconvex}
A subspace $Z$ of a geodesic metric space is $R$-\emph{quasiconvex}, for some constant $R\ge 0$, if every geodesic segment with endpoints on $Z$ is contained in the $R$-neighbourhood of $Z$.
\end{defn}


For the rest of the section, let $(X,\frakS)$ be a HHS. 
\begin{defn}[HQC]
\label{defn:HQC}
A subspace $Y\subseteq X$ is \emph{$\kappa$-hierarchically quasiconvex}, for some $\kappa\colon [0,+\infty)\to [0,+\infty)$, if:
    \begin{itemize}
        \item For every $U\in\frakS$, $\pi_U(Y)$ is $\kappa(0)$-quasiconvex in $\C U$;
        \item \textbf{Realisation:} for every $x\in X$ and every $R\in[0,+\infty)$, if $\dist_U(x,Y)\le R$ for every $U\in\frakS$ then $\dist_X(x,Y)\le \kappa(R)$.
    \end{itemize}
\end{defn}

\begin{rem}\label{rem:nested_HQC}
    It follows from the definition, together with the fact that coordinate projections are coarsely Lipschitz, that if $Y$ and $Z$ are two subspaces of $X$ within Hausdorff distance $d$, and if $Y$ is $\kappa$-hierarchically quasiconvex, then $Z$ is $\kappa'$-hierarchically quasiconvex, for some function $\kappa'$ only depending on $\kappa$, $d$, and $(X,\frakS)$. This observation will be used repeatedly throughout the section.
\end{rem}

An equivalent definition of hierarchical quasiconvexity, which more closely resembles Definition~\ref{defn:quasiconvex}, involves being closed under certain quasigeodesic paths, called “hierarchy paths”:
\begin{defn}[Hierarchy path]
    For $\lambda \ge 1$, a (not necessarily continuous) path $\gamma\colon [a,b]\subset \mathbb{R}\to X$ is a $\lambda$–hierarchy path if
    \begin{itemize}
        \item $\gamma$ is a $(\lambda, \lambda)$-quasigeodesic,
        \item for each $W\in\frakS$, the path $\pi_w(\gamma)$ is an unparameterised $(\lambda,\lambda)$-quasigeodesic, meaning that it becomes a $(\lambda,\lambda)$-quasigeodesic after precomposing it with an increasing function $g\colon[0,l]\to[a,b]$ mapping $0$ to $a$ and $l$ to $b$.
    \end{itemize}
\end{defn}

\begin{rem}\label{rem:existence_of_hpath}
    \cite[Theorem 5.4]{HHS_I} states that any two points of $X$ are connected by a $\lambda_0$-hierarchy path, for some constant $\lambda_0$ only depending on $(X, \frakS)$. From now on, we will assume that the hierarchical constant $E$ has been chosen greater than this $\lambda_0$.
\end{rem}

\begin{lem}[{\cite[Theorem 5.7]{RST}}]\label{lem:equivalent_def_HQC}
    A subset $Y\subseteq X$ is $\kappa$–hierarchically quasiconvex if and only if there exists a function $\eta\colon [1,+\infty)\to [0,+\infty)$ such that every $\lambda$-hierarchy path with endpoints on $Y$ is contained in the $\eta(\lambda)$-neighbourhood of $Y$. Moreover, $\kappa$ and $\eta$ each determine the other.
\end{lem}

The following combines \cite[Definition 5.4 and Lemma 5.5]{HHS_II}:
\begin{defn}[Gate]\label{defn:gate} Let $Y\subseteq X$ be $\kappa$-hierarchically quasiconvex. There exists a coarsely Lipschitz, coarse retraction $\gate_Y\colon X\to Y$, called the \emph{gate} on $Y$, such that, for every $x\in X$ and every $W\in \frakS$, $\pi_W(\gate_Y(x))$ uniformly coarsely coincides with the coarse closest point projection of $\pi_W(x)$ to the $\kappa(0)$-quasiconvex subset $\pi_W(Y)\subseteq \C W$.
\end{defn}

Given two HQC subspaces $A,B$, the gate of one onto the other can be characterised in terms of the \emph{coarse intersection} of $A$ and $B$:

\begin{lem}[{\cite[Lemma 4.10]{quasiflat}}]\label{lem:gate=intersection}
    Let $A$ and $B$ be $\kappa$-hierarchically quasiconvex subsets of $X$, and let $\dist_X(A,B)=r$. There exist non-negative constants $R_0$ and $D_1$, both depending only on $\kappa$ and $r$, such that $\dist_{Haus}(N_{R_0}(A)\cap N_{R_0}(B), \gate_B(A))\le D_1$.
\end{lem}

Furthermore, if $G$ is a group and $A$, $B$ are \emph{subgroups} of $G$, then by \cite[Lemma 4.5]{Hruska_Wise} there exist $D_2$, depending on $R_0$ and on a choice of a word metric for $G$, such that $\dist_{Haus}(N_{R_0}(A)\cap N_{R_0}(B), A\cap B)\le D_2.$ Thus, for HQC subgroups of a HHG, the gate of one onto the other is within finite Hausdorff distance from the actual intersection:

\begin{lem}\label{lem:gate=intersection_for_groups}
    Let $A$ and $B$ be $\kappa$-hierarchically quasiconvex subgroups of a HHG $(G,\frakS)$. There exists a constant $L$, which only depends on $(G,\frakS)$ and $\kappa$, such that $\dist_{Haus}(A\cap B, \gate_B(A))\le L.$
\end{lem}

\subsection{Hierarchical quasiconvexity of a union}\label{subsec:aub}
In hyperbolic spaces, a union of quasiconvex subspaces is again quasiconvex. We provide a proof of this fact for completeness:
\begin{lem}\label{lem:union_of_quasiconvex_in_hyp}
    Let $X$ be a $\delta$-hyperbolic space, and let $Y,Y'$ be two $R$-quasiconvex subspaces. Then $Y\cup Y'$ is $\widehat R$-quasiconvex, where $$\widehat R=R+2\delta+\dist_X(Y,Y')+1.$$ 
\end{lem}

\begin{proof}
    Let $\gamma$ be a geodesic segment with endpoints $a,b\in Y\cup Y'$. If $a$ and $b$ are both in $Y$, or both in $Y'$, then $\gamma$ is in the $R$-neighbourhood of $Y\cup Y'$, by quasiconvexity of each subset. Thus, suppose without loss of generality that $a\in Y$ and $b\in Y'$. Let $p\in Y$ and $p'\in Y'$ be such that $\dist_X(p,p')\le \dist_X(Y,Y')+1$, and choose geodesics $[a,p]$, $[p, p']$, and $[p', b]$. Then $\gamma$ lies in the $2\delta$-neighbourhood of $[a,p]\cup [p, p']\cup [p', b]$, as geodesic quadrangles are $2\delta$-slim in hyperbolic spaces. In turn, $[a,p]\cup [p, p']\cup [p', b]$ is contained in the $(R+\dist_X(Y,Y')+1)$-neighbourhood of $Y\cup Y'$, and the conclusion follows.
\end{proof}

Now we want to establish an analogue of Lemma~\ref{lem:union_of_quasiconvex_in_hyp} for HQC subspaces of a HHS. More precisely, we shall prove that the necessary and sufficient condition for the union to be HQC is the following. Recall that, given two subspaces $C\subseteq A$ of a metric space $X$, we say that $C$ is $R$-dense in $A$ if $A\subseteq N_R(C)$.

\begin{defn}\label{defn:fill_all_gaps}
    Let $(X,\frakS)$ be a HHS, and let $A,B\subseteq X$ be hierarchically quasiconvex subspaces. We say that $A$ and $B$ \emph{fill all squares} if there exists a constant $T$ such that, whenever $U,V\in\frakS$ are orthogonal, either $\pi_U(\gate_A(B))$ is $T$-dense in $\pi_U(A)$, or $\pi_V(\gate_B(A))$ is $T$-dense in $\pi_V(B)$. 
\end{defn}
Notice that, if $A$ and $B$ fill all squares for some constant $T$, then they also fill all squares for any bigger constant $T'\ge T$.

\begin{rem}\label{rem:L-shape}
Definition~\ref{defn:fill_all_gaps} roughly forbids the following situation, which gives the name to the property. Let $X=\mathbb{R}^2$ with the HHS structure coming from the usual Cartesian coordinates, let $A$ be the $x$-axis and $B$ the $y$-axis. Both $A$ and $B$ are hierarchically quasiconvex, as they correspond to factors of the product structure. However, $A\cup B$ is not hierarchically quasiconvex, as it does not satisfy the realisation property: any point $p\in\mathbb{R}^2$ has the same $x$-coordinate as some point in $A$ and the same $y$-coordinate as some point in $B$, but can be arbitrarily far from $A\cup B$. In other words, $A$ and $B$ leave some gaps in the square. 
\end{rem}

%

%

\begin{thm}\label{lem:AUB}
    Let $(X,\frakS)$ be a HHS, and let $A,B\subseteq X$ be $\kappa$-hierarchically quasiconvex subspaces. Then $A$ and $B$ fill all squares (Definition~\ref{defn:fill_all_gaps}), for some constant $T\ge 0$, if and only if $A\cup B$ is $\widetilde\kappa$-hierarchically quasiconvex, where $T$ and $\widetilde\kappa$ each determine the other (together with $\kappa$ and $\dist_X(A,B)$).
\end{thm}

We split the two implications of Theorem~\ref{lem:AUB} into Lemmas~\ref{lem:AUB_1} and~\ref{lem:AUB_2} below.

\begin{lem}\label{lem:AUB_1}
    Let $(X,\frakS)$ be a HHS, and let $A,B\subseteq X$ be $\kappa$-hierarchically quasiconvex subspaces. If $A\cup B$ is $\widetilde\kappa$-hierarchically quasiconvex then $A$ and $B$ fill all squares (Definition~\ref{defn:fill_all_gaps}), for some constant $T$ which only depends on $(X,\frakS)$, $\kappa$, and $\widetilde\kappa$.
    \end{lem}
    
\begin{proof} Fix a hierarchical constant $E$ for $(X,\frakS)$. Assume by contradiction that $A$ and $B$ do not fill all squares. This means that, for every $T\in[0,+\infty)$, one can find two orthogonal domains $U,V\in \frakS$ such that $\min\{\diam_U(A), \diam_V(B)\}\ge T$, but neither $\pi_U(\gate_A(B))$ is $T$-dense in $\pi_U(A)$ nor $\pi_V(\gate_B(A))$ is $T$-dense in $\pi_V(B)$. 

From now on, we will say that a quantity is \emph{uniform} if it does not depend on $T$, but only on $(X,\frakS)$, $\kappa$, and $\widetilde\kappa$. We shall eventually choose $T$ greater than every uniform constant we will find along the way, and this will yield the desired contradiction.

Since $\pi_U(\gate_A(B))$ is not $T$-dense in $\pi_U(A)$, there exists a point $a\in A$ such that $\dist_U(a,\gate_A(B))> T$. Similarly, choose $b\in B$ such that $\dist_V(b,\gate_B(A))> T$. Now consider the following tuple of coordinates:
    \begin{equation}\label{equation:xW}
    x_W=\begin{cases}
        \pi_W(a)\mbox{ if } W\nest U;\\
        \pi_W(b)\mbox{ if } W\nest V;\\
        \pi_W(a)\mbox{ if } W\orth U \mbox{ and } W\orth V;\\
        \rho^U_W\cup \rho^V_W\mbox{ otherwise,}
    \end{cases}        
    \end{equation}
    where, in order to make the definition more compact, we slightly abused the notation by setting $\rho^U_W=\emptyset$ if $U\orth W$, and similarly for $V$. Notice that the tuple $\{x_W\}_{W\nest U}$ is $E$-consistent, as $a$ is a point of $X$ and therefore satisfies the consistency inequalities by Axiom~\eqref{item:dfs_transversal}, and for the same reason $\{x_W\}_{W\nest V}$ and $\{x_W\}_{W\orth U, W\orth V}$ are $E$-consistent as well. Arguing as in \cite[Construction 5.10]{HHS_II}, one can show that the whole tuple $\{x_W\}_{W\in\frakS}$ is $K_1$-consistent for some uniform constant $K_1$. Then by the Realisation Theorem~\ref{thm:realisation} we can find $p\in X$ and a uniform constant $K_2$ such that $\dist_W(p,x_W)\le K_2$ for every $W\in\frakS$.

    \begin{claim}\label{claim:p_close_to_AuB}
        There exist uniform constants $D\ge 0$ and $K_3\ge K_2$ such that, if $T\ge D$, then $\dist_W(p,A\cup B)\le K_3$ for every $W\in\frakS$.
    \end{claim}
    \begin{proof}
        We refer to how we defined $x_W$ in Equation~\eqref{equation:xW}. If either $W$ is nested into $U$ or $V$, or $W$ is orthogonal to both, then $\pi_W(p)$ is $K_2$-close to either $\pi_W(a)$ or $\pi_W(b)$, and we have nothing to prove. Thus suppose, without loss of generality, that either $U\propnest W$ or $U\transverse W$. Then $\pi_W(p)$ is $K_2$-close to $\rho^U_W\cup \rho^V_W$, and in particular 
        $$\dist_W(p,U)\le K_2+\diam_W(\rho^V_W)+\dist_W(U,V)\le K_2+3E,$$
        where we used that $\dist_W(U,V)\le 2E$ by Lemma~\ref{lem:close_proj_of_orthogonals}. Now let $P_U$ be the \emph{product region} associated to $U$, as defined in \cite[Definition 5.15]{HHS_II}. For our purposes, $P_U$ can be thought of as the subspace of all $z\in X$ such that, for every $Y\in\frakS$ such that $U\propnest Y$ or $U\transverse Y$, the projection $\pi_Y(z)$ coincides with $\rho^U_Y$, up to some error which is bounded in terms of $E$. By \cite[Proposition 4.24]{RST}, there exists uniform constants $D,\lambda,\nu\ge 1$ such that, if $T\ge D$, then there exists a $\lambda$-hierarchy path $\gamma$ connecting $a$ to $\gate_A(B)\subseteq A$ which passes $\nu$-close to the product region $P_U$. Since $A$ is $\kappa$-hierarchically quasiconvex, by Lemma~\ref{lem:equivalent_def_HQC} we have that $\gamma$ is contained in some uniform neighbourhood of $A$. Thus the distance between $A$ and $P_U$ is uniformly bounded. In turn, since the projection map $\pi_W\colon X\to \C W$ is $E$-coarsely Lipschitz, $\pi_W(A)$ is uniformly close to $\pi_W(P_U)$, which in turn uniformly coarsely coincides with $\rho^U_W\subseteq \pi_W(p)$. 
    \end{proof}
    Since $A\cup B$ is $\widetilde\kappa$-quasiconvex, the realisation property for $A\cup B$ tells us that $p$ is $\widetilde\kappa(K_3)$-close to $A\cup B$, and without loss of generality we can assume that $\dist_X(p,A)\le\widetilde\kappa(K_3)$. Then $\gate_B(p)$ is uniformly close to $\gate_B(A)$, as gate maps are coarsely Lipschitz. However, by how gate maps are constructed (see Definition~\ref{defn:gate}), we have that $\pi_V(\gate_B(p))$ uniformly coarsely coincides with the projection of $\pi_V(p)$ to the quasiconvex subset $\pi_V(B)$, and $\pi_V(p)=\pi_V(b)$ already lies in $\pi_V(B)$.
    Summarising, there is some constant $K_4$, only depending on $(X,\frakS)$, $\kappa$, and $\widetilde\kappa$, such that $\dist_V(b, \gate_B(A))\le K_4$, and this is against our choice of $b$ if we choose $T>K_4$.
\end{proof}

\begin{lem}\label{lem:AUB_2}
       Let $(X,\frakS)$ be a HHS, and let $A,B\subseteq X$ be $\kappa$-hierarchically quasiconvex subspaces. If $A$ and $B$ fill all squares (Definition~\ref{defn:fill_all_gaps}), for some constant $T\ge 0$, then $A\cup B$ is $\widetilde\kappa$-hierarchically quasiconvex, where $\widetilde\kappa$ only depends on $(X,\frakS)$, $\kappa$, $\dist_X(A,B)$, and $T$.
\end{lem}

\begin{proof}
Firstly, by Lemma~\ref{lem:union_of_quasiconvex_in_hyp} the projection $\pi_U(A\cup B)$ is $\widetilde\kappa(0)$-quasiconvex, where $\widetilde\kappa(0)$ is a constant depending on $\kappa(0)$, $D\coloneq \dist_X(A,B)$, and a hierarchical constant $E$ for $(X,\frakS)$. 

Thus we are left to prove the realisation property for $A\cup B$. 
To this purpose, fix $R\ge 0$, let $x\in X$ be such that $\dist_W(x,A\cup B)\le R$ for every $W\in\frakS$, and set $$\Omega=\kappa(0)+R+T+E(D+KD+K+4),$$ 
where $K$ is the constant from \cite[Lemma 1.27]{quasiflat}, depending only on $\kappa$ and $(X, \frakS)$, such that for every $b\in B$
$$\dist_X(b, \gate_A(b))\le KD+K.$$

We claim that there exists $F\in\{A,B\}$ such that $\dist_W(x,F)\le \Omega$ for every $W\in\frakS$. This will imply that $\dist_X(x,A\cup B)\le \kappa(\Omega)$, so we can set $\widetilde\kappa(R)=\kappa(\Omega)$. If this is not the case, let $U,V\in\frakS$ be (necessarily distinct) domains such that $\dist_U(x,B)> \Omega$ and $\dist_V(x,A)> \Omega$. Notice that, as $\Omega\ge R$, we must have that $\dist_U(x,A)\le R$ and $\dist_V(x,B)\le R$. There are a few configurations to analyse, depending on the relation between $U$ and $V$.
    \begin{itemize}
        \item \textbf{$U\transverse V$}: By the Behrstock inequality, Definition~\ref{defn:Hierarchical_space}.\eqref{item:dfs_transversal}, we can assume without loss of generality that $\dist_U(x,\rho^V_U)\le E$. Moreover, since $\Omega\ge 4E$ and $\diam(\rho^V_U)\le E$, we have that $\dist_U(B,\rho^V_U)> 2E$, and again the Behrstock inequality tells us that $\pi_V(B)$ has diameter at most $3E$. Thus 
        $$\dist_V(x,A)\le \dist_V(x,B)+\diam_V(B)+\dist_V(A,B)\le R+3E+E(D+1),$$
        where we used that $\dist_V(A,B)\le E(D+1)$ as the projection $\pi_V\colon X\to \C V$ is $(E,E)$-coarsely Lipschitz. But then $\dist_V(x,A)\le \Omega$, against our assumption.

        \item \textbf{$U\propnest V$}:  Since $\dist_U(x,B)>\Omega\ge E$, the bounded geodesic image Lemma~\ref{lem:bgi} yields that every geodesic connecting $\pi_V(x)$ to $\pi_V(B)$ must pass $E$-close to $\rho^U_V$. Since $\dist_V(x,B)\le R$ we get that $\dist_V(x,U)\le R+E$, and in turn 
        $$\dist_V(U,A)\ge \dist_V(x,A)-\dist_V(x,U)-\diam_V(\rho^U_V)\ge 2E+\kappa(0).$$ 
        Now, every geodesic $\gamma\subseteq \C V$ connecting two points of $\pi_V(A)$ is contained in the $\kappa(0)$-neighbourhood of $\pi_V(A)$, and therefore is at least $2E$-far from $\rho^U_V$. Then again the bounded geodesic image Lemma~\ref{lem:bgi} tells us that $\pi_U(A)$ has diameter at most $E$, and as in the previous case we get that $\dist_U(x,B)\le \Omega$, yielding a contradiction. Notice that the same argument covers the case when $V\propnest U$.
        
        \item \textbf{$U\orth V$}: Since $A$ and $B$ fill all squares, we can assume without loss of generality that $\pi_U(\gate_A(B))$ is $T$-dense in $\pi_U(A)$, so there exists $b\in B$ such that $\dist_U(x,\gate_A(b))\le R+T$. In turn,
        $$\dist_U(x,B)\le \dist_U(x,b)\le \dist_U(x,\gate_A(b))+\dist_U(\gate_A(b), b)\le R+T+E(KD+K+1)\le\Omega,$$
        where again we used that $\pi_U$ is $(E,E)$-coarsely Lipschitz. This again yields a contradiction.
    \end{itemize}
\end{proof}

\subsection{Hierarchical quasiconvexity of an amalgam}\label{subsec:amalgam_hqc} 
We devote the rest of the Section to the proof of Theorem~\ref{thmintro:amalgamation_HQC}, regarding the amalgamation of HQC subgroups. First, a definition.

\begin{defn}[Hierarchy path hull]\label{def:hierarchyhull}
Let $Z$ be a subset of the hierarchically hyperbolic
space $(X,\frakS)$. Define $P^1_\lambda(Z)$ to be the union of all $\lambda$–hierarchy paths between points in $Z$. One can then inductively set $P^n_\lambda(Z)=P^1_\lambda(P^{n-1}_\lambda(Z))$ for all $n\ge 2$. 

Recall that every two points of $X$ are connected by a $\lambda_0$-hierarchy path, where, $\lambda_0$ is the constant from Remark~\ref{rem:existence_of_hpath}; thus, for every $\lambda\ge\lambda_0$ and every $n\ge 1$, the hull $P^n_\lambda(Z)$ is non-empty and contains $Z$. 
\end{defn}

The following restates the core result of \cite{RST}, which was crucial in establishing the equivalence between hierarchical quasiconvexity and being closed under hierarchy paths:
\begin{lem}[{\cite{RST}}]\label{lem:hierarchyhull_HQC}
    There exist $N\in\mathbb{N}$, $\Lambda\ge \lambda_0$, and $\ov\kappa$, all depending only on $(X,\frakS)$,  such that, for every $Z\subseteq X$, the $N$-th hierarchy path hull $P^N_{\Lambda}(Z)$ is $\ov\kappa$-hierarchically quasiconvex.
\end{lem}

\begin{proof}
    For every $\theta\ge 0$, let $H_{\theta}(Z)$ be the $\theta$-quasiconvex hull of $Z$, as defined in e.g. \cite[Definition 4.12]{RST}. By \cite[Lemma 6.2]{HHS_II}, there exists $\theta_0$ such that, for every $Z\subset X$ and every $\theta\ge \theta_0$, the hull $H_{\theta}(Z)$ is $\kappa_\theta$-hierarchically quasiconvex, where $\kappa_\theta$ only depends on $\theta$ and the HHS structure. Moreover, \cite[Theorem 5.2]{RST} states that there exist $N$ and $\Lambda$ as above, and a constant $\ov\theta\ge \theta_0$, such that $\dist_{Haus}(P^N_{\Lambda}(Z), H_{\ov\theta}(Z))\le D$, where $D$ only depends on $\Lambda$, $\ov\theta$, and $(X, \frakS)$. Thus, as $P^N_{\Lambda}(Z)$ is within Hausdorff distance at most $D$ from a $\kappa_{\ov\theta}$-hierarchically quasiconvex set, it is itself $\ov\kappa$-hierarchically quasiconvex by Remark~\ref{rem:nested_HQC}, for some function $\ov\kappa$ depending on $\kappa_{\ov\theta}$ and $D$ (and therefore, ultimately, only on $(X,\frakS)$).
\end{proof}

We now introduce the main technical requirement of Theorem~\ref{thmintro:amalgamation_HQC}. Its necessity will be discussed in Subsection~\ref{subsec:no_drift}, where we also provide an explicit counterexample.
\begin{defn}\label{defn:no_drift} Let $(G,\frakS)$ be a HHG, let $A,B\le G$ be two subgroups satisfying the hypothesis of Theorem~\ref{thm:amalgamation}, and let $C=A\cap B$. We say that $A,B$ have \emph{no drift in the orthogonals} if there exists $R\ge0$ such that the following hold. For every $a,b\in (A\cup B)-C$ belonging to different factors, and for every domain $U\in\frakS$ which is orthogonal to both $a^{-1}Y_a$ and $Y_b$, we have that either $\pi_U(A)$ or $\pi_U(B)$ is $R$-dense in $\C U$.
\end{defn}

\begin{thm}\label{thm:amalgamation_of_quasiconvex}
    Let $(G,\frakS)$ be a HHG, let $A,B\le G$ be two $\kappa$-hierarchically quasiconvex subgroups, and let $C=A\cap B$. Suppose that:
    \begin{itemize}
        \item There exists $M\ge100E$ such that $A$ and $B$ satisfy the hypotheses of Theorem~\ref{thm:amalgamation};
        \item $A$ and $B$ fill all squares (Definition~\ref{defn:fill_all_gaps}), for some constant $T\ge 0$;
        \item $A$ and $B$ have no drift in the orthogonals (Definition~\ref{defn:no_drift}), for some constant $R\ge 0$.
    \end{itemize}
    There exist a positive constant $\mathfrak M$ and a function $\mathfrak k\colon[0,+\infty)\to [0,+\infty)$, both depending only on $\kappa$, $T$, $R$, and $(G,\frakS)$, such that, if $M\ge \mathfrak M$, then $\langle A,B\rangle_G\cong A*_C B$ is $\mathfrak k$-hierarchically quasiconvex in $G$.
\end{thm}

Theorem~\ref{thm:amalgamation_of_quasiconvex} is a consequence of the following technical statement:
\begin{lem}\label{lem:nbh_chiuso_per_hp}
    There exist functions $K'\colon[0,+\infty)\to [0,+\infty)$ and $M\colon[0,+\infty)\to [100E,+\infty)$, depending only on $\kappa$, $T$, $R$, and $(G,\frakS)$, such that the following holds. Let $\Lambda$ be as in Lemma~\ref{lem:hierarchyhull_HQC}. For every $K\ge 0$, if $M\ge M(K)$ then 
    $$P^1_{\Lambda}(N_K(A*_C B))\subseteq N_{K'(K)}(A*_C B).$$
\end{lem}

\begin{proof}[Proof of Theorem~\ref{thm:amalgamation_of_quasiconvex}, assuming Lemma~\ref{lem:nbh_chiuso_per_hp}] Set $K_0=0$, and iteratively define $M_n=M(K_{n-1})$ and $K_n=K'(K_{n-1})$ for every $n=1,\ldots, N$, where $N$ is the integer from Lemma~\ref{lem:hierarchyhull_HQC}. Now let $\mathfrak M=\max_{i=1,\ldots, N} M_n$. If $M\ge \mathfrak M$, then 
$$P^N_{\Lambda}(A*_C B)\subseteq P^{N-1}_{\Lambda}(N_{K_1}(A*_C B))\subseteq\ldots\subseteq N_{K_N}(A*_C B).$$
Now, $P^N_{\Lambda}(A*_C B)$ is $\ov \kappa$-HQC, where $\ov\kappa$ is the function from Lemma~\ref{lem:hierarchyhull_HQC} which only depends on $(G,\frakS)$. Furthermore, as $P^N_{\Lambda}(A*_C B)$ and $A*_C B$ are within Hausdorff distance at most $K_N$, Remark~\ref{rem:nested_HQC} implies that $A*_C B$ is itself $\mathfrak k$-hierarchically quasiconvex, for some function $\mathfrak k$ only depending on $K_N$ and $\ov\kappa$. As by hypothesis $K_N$ only depends on $\kappa$, $T$, $R$, and $(G,\frakS)$, this concludes the proof of Theorem~\ref{thm:amalgamation_of_quasiconvex}.
\end{proof}

\begin{proof}[Proof of Lemma~\ref{lem:nbh_chiuso_per_hp}]
Recall that, if $G$ is a HHG, we can assume that the underlying hierarchically hyperbolic space $X$ is a Cayley graph for $G$, with respect to any finite generating set. Let $\gamma$ be a $\Lambda$-hierarchy path between points of $N_K(A*_C B)$. If we connect each endpoint of $\gamma$ with any point of $A*_C B$ within distance $K$, we get a $\lambda'$-hierarchy path $\gamma'$ which extends $\gamma$, for some constant $\lambda'\ge \Lambda$ depending only on $K$, $\Lambda$, and the hierarchical constant for $G$. Up to the action of $A*_C B$ on itself, we can assume that the endpoints of $\gamma'$ are the identity element $1$ and some $w=g_1\ldots g_k c$, where $c\in C$, $g_i\not \in C$ for all $i=1,\ldots, k$, and every two consecutive $g_i$ and $g_{i+1}$ belong to different factors of the amalgamation. 

\begin{notation}\label{notation:ACW}
$A$ and $B$ satisfy the hypothesis of Theorem~\ref{thm:amalgamation}, whose data include a basepoint $x_0\in G$ and a domain $Y_a\in\frakS$ for every $a\in (A\cup B)-C$. As in Notation~\ref{notation:C_iW_i}, for every $i=1,\ldots, k$ let $Y_i=Y_{g_i}$, and set 
$$C_i=g_1\ldots g_{i-1}Cx_0,\quad W_i=g_1\ldots g_{i-1}Y_i.$$
Moreover, for every $i=1,\ldots, k$ let 
$$H_i=\begin{cases}
    g_1\ldots g_{i-1} A\mbox{ if }g_i\in A;\\
g_1\ldots g_{i-1} B\mbox{ if }g_i\in B.
\end{cases}$$ 
This way, $H_i$ contains both $C_i$ and $C_{i+1}$.
\end{notation}

\begin{notation}\label{notation:choice_of_M}
\newcommand{\MK}{200E+10\lambda'+T+LE+20S}
We shall prove Lemma~\ref{lem:nbh_chiuso_per_hp} assuming that $M\ge M(K)$, where
$$M(K)=\MK.$$ 
In the expression above, $L=L(\kappa, \frakS)$ is the constant from Lemma~\ref{lem:gate=intersection_for_groups}, while the constant $S=S(E,\kappa(0),\lambda',\frakS)$ will be defined in the proof of Claim~\ref{claim:li_Ai_close} below (more precisely, in the paragraph named \textbf{Case 1}).
\end{notation}

Now fix any $i=2, \ldots, k-1$. Recall that $\dist_{W_i}(C_i, C_{i+1})\ge M$ by Assumption~\ref{hyp::4} of Theorem~\ref{thm:amalgamation}. Combining this with Lemma~\ref{lem:Cj_close_in_W_i}, we get that
\begin{equation}
    \dist_{W_i}(1,w)\ge \dist_{W_i}(C_i, C_{i+1})-\dist_{W_i}(C_i, 1)-\dist_{W_i}(C_{i+1}, w)\ge 4M/5-10E.
\end{equation}
Notice that, by our choice of $M$ in Notation~\ref{notation:choice_of_M}, we have that $$4M/5-10E\ge 2M/5+2\lambda'.$$ In other words, the balls of radius $M/5$ around $\pi_{W_i}(1)$ and $\pi_{W_i}(w)$ are at distance at least $2\lambda'$. Then, as the projection of $\gamma'$ to $\C W_i$ is a $(\lambda', \lambda')$-quasigeodesic (after reparametrisation), there must be a point $\ell_i\in\gamma'$ such that 
\begin{equation}\label{equation:dist_1,ell}
    \min\{\dist_{W_i}(1,\ell_i), \dist_{W_i}(w,\ell_i)\}>M/5.
\end{equation}

Now, the core of the proof is the following:

\begin{claim}\label{claim:li_Ai_close}
     For every $i=2\ldots, k-1$, the distance between $\ell_i$ and $H_i$ is bounded by some constant $\Psi$, depending on $K$, $\kappa$, $T$, $R$, and $(G,\frakS)$.
\end{claim}
Before proving the Claim, we show that it implies Lemma~\ref{lem:nbh_chiuso_per_hp}. Indeed, we can decompose $\gamma'$ as a union of $\lambda'$-hierarchy paths $\gamma'_i$ with endpoints $\{\ell_i, \ell_{i+1}\}$, where we set $\ell_1=1$ and $\ell_{k}=w$. Every $\gamma'_i$ thus connects two points on some coset of $N_{\Psi}(A\cup B)$. As $A$ and $B$ $T$-fill all squares, by Theorem~\ref{lem:AUB} there exists $\widetilde \kappa$, depending on $\kappa$ and $T$, such that $A\cup B$ is $\widetilde\kappa$-hierarchically quasiconvex; then Remark~\ref{rem:nested_HQC} implies that $N_{\Psi}(A\cup B)$ is $\kappa'$-hierarchically quasiconvex, where $\kappa'$ depends on $\widetilde \kappa$ and $\Psi$. Finally, Lemma~\ref{lem:equivalent_def_HQC} implies that each $\gamma'_i$ is contained in a neighbourhood of a coset of $N_{\Psi}(A\cup B)$, whose radius only depends on $\kappa'$ and $\lambda'$. Summing everything up, we just proved that, if $M\ge M(K)$, then $\gamma'$ decomposes as a union of subpaths, whose distance from $A *_C B$ is bounded above only in terms of $K$, $\kappa$, $T$, $R$, and $(G,\frakS)$.

\begin{figure}[htp]
\centering
\includegraphics[width=\textwidth]{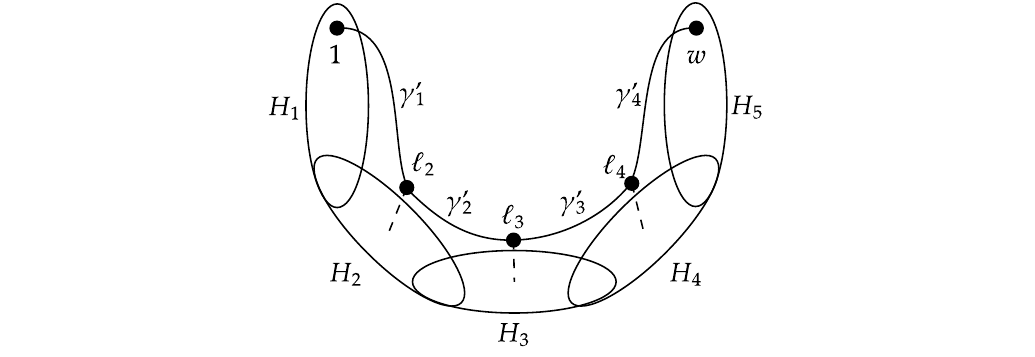}
\caption{The hierarchy path $\gamma'$ must pass $\Psi$-close to each $H_i$ (here denoted by the dashed lines). In other words, $\gamma'$ decomposes as a union of sub-hierarchy paths whose endpoints are uniformly close to some coset of $A\cup B$, and we can invoke Lemma~\ref{lem:AUB} to get that $\gamma'$ lies in a uniform neighbourhood of $A *_C B$.}
\label{fig:catena_degli_li}
\end{figure}

\begin{proof}[Proof of Claim \ref{claim:li_Ai_close}]
Recall that, for every $i=2,\ldots, k-1$, we have a point $\ell_i$ such that $$\min\{\dist_{W_i}(1,\ell_i), \dist_{W_i}(w,\ell_i)\}>M/5,$$ and we want to prove that $\ell_i$ and $H_i$ are uniformly close. To do so, it is enough to prove that, for every $U\in\frakS$, $\dist_U(H_i, \ell_i)$ is uniformly bounded in terms of  $K$, $\kappa$, $T$, $R$, and $(X,\frakS)$, because then we can apply the realisation property of the $\kappa$-HQC $H_i$. There are five cases to analyse, depending on the relation between $U$ and $W_i$.

\par\medskip
\textbf{Case 1: $U=W_i$.} Let $c\in C_i$ be such that $$\dist_{W_i}(1,c)\le \dist_{W_i}(1,C_i)+E\le M/10+6E,$$ where we invoked Lemma~\ref{lem:Cj_close_in_W_i}. Similarly, let $c'\in C_{i+1}$ be such that $$\dist_{W_i}(w,c')\le \dist_{W_i}(w,C_{i+1})+E\le M/10+6E.$$ Fix three geodesics $[\pi_{W_i}(1),\pi_{W_i}(c)]\cup [\pi_{W_i}(c),\pi_{W_i}(c')]\cup [\pi_{W_i}(c'),\pi_{W_i}(w)] \subset \C W_i$. As $\C W_i$ is $E$-hyperbolic and $\pi_{W_i}(\gamma')$ is a $(\lambda', \lambda')$-quasigeodesic (after reparametrisation), there exists a constant $S'$, depending on $E$ and $\lambda'$, such that $\pi_{W_i}(\ell_i)$ is $S'$-close to one of the three geodesics (this is a consequence of e.g. \cite[Theorem III.1.7]{bridsonhaefliger}, plus the fact that geodesic quadrangles in $E$-hyperbolic spaces are $2E$-thin). 

Now, if $\pi_{W_i}(\ell_i)$ is $(2S'+6E)$-close to $[\pi_{W_i}(c),\pi_{W_i}(c')]$, then by $\kappa(0)$-quasiconvexity of $H_i$ we have that $\dist_{W_i}(\ell_i, H_i)\le S$, where $S\coloneq 2S'+6E+ \kappa(0)$ (this is the constant we use in Notation~\ref{notation:choice_of_M} to choose $M$). 

Thus suppose by contradiction that $\pi_{W_i}(\ell_i)$ is at least $(2S'+6E)$-far from the geodesic $[\pi_{W_i}(c),\pi_{W_i}(c')]$, and without loss of generality we can assume that $\pi_{W_i}(\ell_i)$ is $S'$-close to some point $r$ on the geodesic $[\pi_{W_i}(1),\pi_{W_i}(c)]$, as in Figure~\ref{fig:U=Wi}. In particular $\dist_{W_i}(r,c)>S'+6E$, because otherwise $\pi_{W_i}(\ell_i)$ would be at distance at most $(2S'+6E)$ from $\pi_{W_i}(c)$. But then
$$\dist_{W_i}(1,\ell_i)\le \dist_{W_i}(1,r)+\dist_{W_i}(r,\ell_i)=$$
$$=\dist_{W_i}(1,c)- \dist_{W_i}(r,c)+\dist_{W_i}(r,\ell_i)<$$
$$< (M/10+6E)- (S'+6E)+S'=M/10<M/5,$$
contradicting Equation\eqref{equation:dist_1,ell}.

\begin{figure}[htp]
    \centering
    \includegraphics[width=\textwidth]{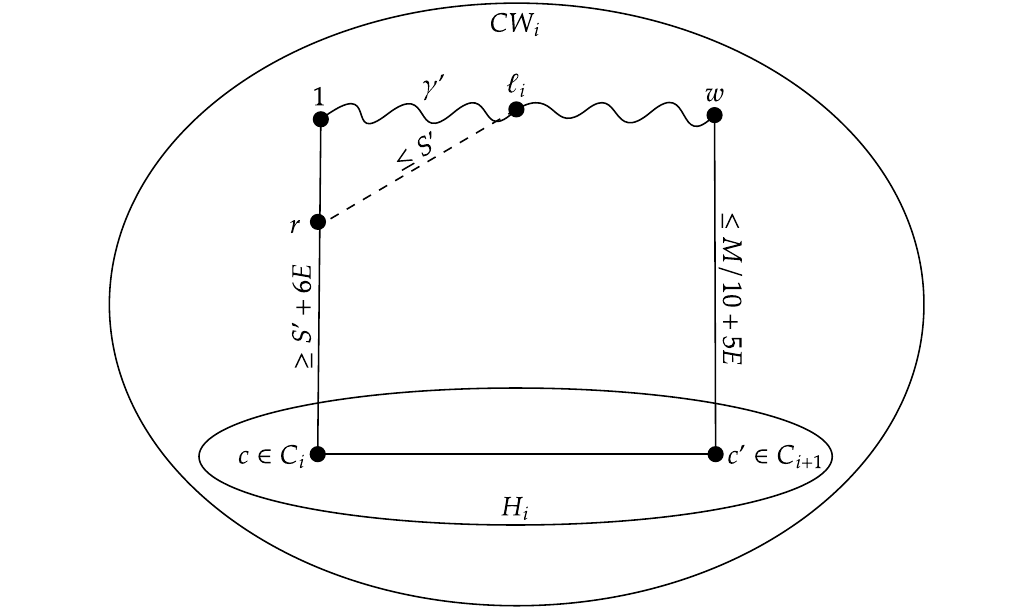}
    \caption{If $\ell_i$ projected too far from $H_i$ in $\C W_i$, then it would also project too close to $1$, contradicting our choice of $\ell_i$.}
    \label{fig:U=Wi}
\end{figure}


\par\medskip
\textbf{Case 2: $U\propnest W_i$.} The projection $\rho^U_{W_i}$ is well-defined. Furthermore, by Case 1 there is some $q\in H_i$ such that $\dist_{W_i}(\ell_i, q)\le S$.

Suppose first that $\dist_{W_i}(\ell_i, U)> S+2E$. Then any geodesic $[\pi_{W_i}(\ell_i), \pi_{W_i}(q)]$ inside $\C W_i$ cannot pass through the $E$-ball around $\rho^U_{W_i}$, so the bounded geodesic image Lemma~\ref{lem:bgi} tells us that $\dist_U(H_i, \ell_i)\le \dist_U(q, \ell_i)\le E$, and we are done.

Thus suppose that $\dist_{W_i}(\ell_i, U)\le S+2E$. As $\dist_{W_i}(1,\ell_i)>M/5$, the triangle inequality yields that 
$$\dist_{W_i}(1,U)\ge \dist_{W_i}(1,\ell_i)-\dist_{W_i}(\ell_i, U)-\diam_{W_i}(\rho^U_{W_i})> M/5-S-3E=$$
$$=M/10+5E + (M/10-S-8E)\ge \dist_{W_i}(1, C_i) + 2E,$$
where we used that, by our choice of $M$ in Notation~\ref{notation:choice_of_M}, $M/10\ge S+10E$, and that $\dist_{W_i}(1, C_i)\le M/10+5E$ by Lemma~\ref{lem:Cj_close_in_W_i}. This means that any geodetic connecting $\pi_{W_i}(1)$ to the closest point in $\pi_{W_i}(C_i)$ cannot pass $E$-close to $\rho^{U}_{W_i}$, and the bounded geodesic image Lemma~\ref{lem:bgi}, applied to the domains $U\propnest W_i$, yields that $\dist_U(1, C_i)\le E$. Symmetrically, one gets that $\dist_U(w, C_{i+1})\le E$. The situation in $\C U$ is depicted in Figure~\eqref{fig:U_inside_Wi}. 

\begin{figure}[htp]
\centering
\includegraphics[width=\textwidth]{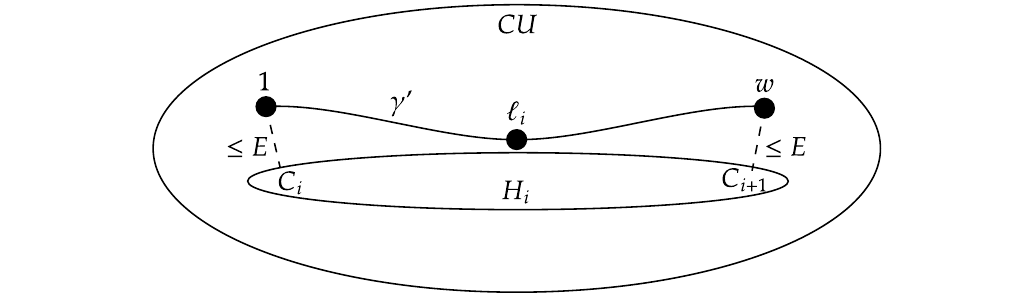}
\caption{The hierarchy path $\gamma'$ projects to a quasigeodesic (after reparametrisation) whose endpoints are in the $E$-neighbourhood of the quasiconvex set $\pi_U(H_i)$.}
\label{fig:U_inside_Wi}
\end{figure}

Now, the projection of $H_{i}$ inside $\C U$ is a $\kappa(0)$-quasiconvex subset, and $\pi_U(\gamma')$ is a $(\lambda', \lambda')$-quasigeodesic (after reparametrisation) whose endpoints $\pi_U(1)$ and $\pi_U(w)$ are within distance at most $E$ from $\pi_U(H_i)$. Therefore, again as a consequence of e.g. \cite[Theorem III.1.7]{bridsonhaefliger}, the distance between $\pi_U(\ell_i)$ and $\pi_U(H_i)$ is bounded in terms of $\kappa(0)$, $\lambda'$, and $E$.

\par\medskip
\textbf{Case 3: $U\transverse W_i$.} We claim that both $H_i$ and $\ell_i$ project uniformly close to $\rho^{W_i}_U$ inside $\C U$. Indeed, $\dist_{W_i}(C_i,C_{i+1})\ge M\ge 4E$, and in particular one between $C_i$ and $C_{i+1}$ is $E$-far from $\rho^U_{W_i}$. Then the Behrstock inequality yields $\dist_{U}(H_i, W_i)\le E$.

In order to bound $\dist_{U}(\ell_i, W_i)$, first notice that the projections of $1$, $\ell_i$, and $w$ to $\C W_i$ are all at distance at least $M/5\ge 4E$ from each other. Thus, again by Behrstock inequality, at least two of these points project $E$-close to $\rho^{W_i}_U$ inside $\C U$. If $\dist_U(\ell_i, W_i)\le E$ we are done; otherwise
$$\dist_U(1, w)\le \dist_U(1, W_i)+\diam_U(\rho^{W_i}_U)+\dist_U(W_i, w) \le 3E.$$
As $\pi_U(\ell_i) $ lies on a $(\lambda',\lambda')$-quasigeodesic between $\pi_U(1) $ and $\pi_U(w) $, we get that $\dist_U(\ell_i, W_i)$ is bounded in terms of $\lambda'$ and $\dist_U(1, w)\le 3E$. 

\par\medskip
\textbf{Case 4: $W_i\propnest U$.} Again, we claim that both $H_i$ and $\ell_i$ project uniformly close to $\rho^{W_i}_U$ inside $\C U$. As pointed out above $\dist_{W_i}(C_i, C_{i+1})\ge 4E$, so by the bounded geodesic image Lemma~\ref{lem:bgi} every geodesic in $\C U$ with endpoints on $\pi_U(C_i)$ and $\pi_U(C_{i+1})$ must pass $E$-close to $\rho^{W_i}_U$. But $\pi_U(H_i)$ is $\kappa(0)$-quasiconvex, hence $\dist_U(W_i, H_i)\le E+\kappa(0)$.

Now, both $\dist_{W_i}(1, \ell_i)$ and $\dist_{W_i}(\ell_i, w)$ are greater than $M/5 \ge E$. Thus any two geodesics $[\pi_{U}(1), \pi_{U}(\ell_i)]$ and $[\pi_{U}(\ell_i), \pi_{U}(w)]$ inside $\C U$ must pass $E$-close to $\rho^{W_i}_U$. In turn, $\pi_U(\gamma')$ is a $(\lambda',\lambda')$-quasigeodesic (after reparametrisation); therefore, again by e.g. \cite[Theorem III.1.7]{bridsonhaefliger}, there exists a constant $\omega\ge 0$, depending only on $\lambda'$ and $E$, such that the segment of $\pi_U(\gamma')$ between $\pi_U(1)$ and $\pi_U(\ell_i)$ is $\omega$-close to the geodesic $[\pi_{U}(1), \pi_{U}(\ell_i)]$. Summing the two facts, we can find a point $p\in  \pi_U(\gamma')$ such that $\dist_U(p,W_i)\le \omega+E$. Arguing similarly for the other segment of $\pi_U(\gamma')$, we get a point $q\in \pi_U(\gamma')$ such that $\dist_U(q,W_i)\le \omega+E$. Now, $\pi_U(\ell_i)$ lies on the segment of $\pi_U(\gamma')$ between $p$ and $q$, as in Figure \ref{fig:Wi_inside_U}, and $\dist_{U}(p,q)\le 2\omega+3E$. This implies that the distance between $\pi_U(\ell_i)$ and $\rho^{W_i}_U$ is controlled in terms of $E$, $\omega$, and $\lambda'$.

\begin{figure}[htp]
\centering
\includegraphics[width=\textwidth]{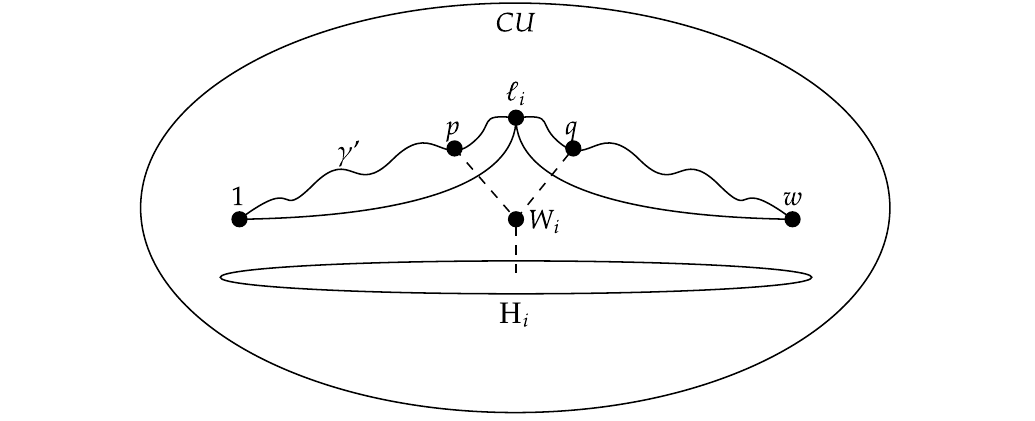}
\caption{As the distance between $p$ and $q$ is uniformly bounded, so is also the length of the $(\lambda', \lambda')$-sub-quasigeodesic of $\pi_U(\gamma')$ between them, on which $\pi_U(\ell_i)$ lies. Thus $\ell_i$ is uniformly close to $\rho^{W_i}_U$, and in turn to $\pi_U(H_i)$.}
\label{fig:Wi_inside_U}
\end{figure}

\par\medskip
\textbf{Case 5: $U\orth W_i$.} This case is itself split into several subcases, as it also involves the relation between $U$ and both $W_{i-1}$ and $W_{i+1}$. 

\textbf{Case 5.1:} First, notice that neither $W_{i-1}$ nor $W_{i+1}$ can be nested in $U$, as otherwise one of them would be orthogonal to $W_i$.

\textbf{Case 5.2:} Suppose that $W_{i-1}$ is also orthogonal to $U$. As $A$ and $B$ have no drift in the orthogonals (Definition~\ref{defn:no_drift}), one of the following happens:
\begin{itemize}
    \item If $\pi_U(H_i)$ is $R$-dense in $\C U$, then in particular $\dist_{U}(\ell_i, H_i)\le R$, and we are done.
    \item Otherwise, $\pi_U(H_{i-1})$ is $R$-dense in $\C U$. If $\diam_U(H_{i-1})\le T$ then $\C U$ has diameter at most $T+2R$, and again we conclude as $\dist_{U}(\ell_i, H_i)\le R+2T$. Otherwise, as $H_{i-1}$ and $H_i$ fill all squares, we must have that $\pi_U(\gate_{H_{i-1}}(H_i))$ is $T$-dense in $\pi_U(H_{i-1})$, and therefore $(T+R)$-dense in $\C U$. Then, as $\pi_U(C_i)$ coarsely coincides with $\pi_U(\gate_{H_{i-1}}(H_i))$ by Lemma~\ref{lem:gate=intersection_for_groups}, we get that the distance between $\pi_U(\ell_i)$ and $\pi_U(C_i)\subseteq \pi_U(H_i)$ is uniformly bounded.
\end{itemize}


\textbf{Case 5.3:} We are left with the cases when both $\rho^{U}_{W_{i-1}}$ and $\rho^{U}_{W_{i+1}}$ are well-defined. First, we notice that
$$M\le \dist_{W_{i-1}}(C_{i-1},C_i)\le \dist_{W_{i-1}}(C_{i-1},U)+\diam_{W_{i-1}}(\rho^U_{W_{i-1}}\cup \rho^{W_i}_{W_{i-1}}) +\dist_{W_{i-1}}(C_i, W_i)\le$$
$$\le \dist_{W_{i-1}}(C_{i-1},U)+4E +M/10+E,$$
where we used Lemma~\ref{lem:close_proj_of_orthogonals} to bound the projections of the two orthogonal domains $U\orth W_i$, and Claim~\ref{claim:projection of adjacent Wi} to bound $\dist_{W_{i-1}}(C_i, W_i)$. Therefore $$\dist_{W_{i-1}}(C_{i-1},U)\ge 9/10M-5E.$$
Furthermore $\dist_{W_{i-1}}(1,C_{i-1})\le M/10+5E$ by Lemma~\ref{lem:Cj_close_in_W_i}, so
$$\dist_{W_{i-1}}(1,U)\ge \dist_{W_{i-1}}(C_{i-1},U)-\diam_{W_{i-1}}(C_{i-1})-\dist_{W_{i-1}}(1,C_{i-1})\ge$$
$$\ge9/10M-5E-M/10-M/10-5E=7/10M-10E.$$
Notice that both $\dist_{W_{i-1}}(1,U)$ and $\dist_{W_{i-1}}(C_{i-1},U)$ are greater than $2E$, by our choice of $M$ in Notation~\ref{notation:choice_of_M}.

Now we claim that $\dist_U(1, C_{i-1})\le 3E$. Indeed, if $W_{i-1}\transverse U$, then the Behrstock inequality~\eqref{item:dfs_transversal} yields that both $\pi_U(1)$ and $\pi_U(C_{i-1})$ are $E$-close to $\rho^{W_{i-1}}_U$. If instead $U\propnest W_{i-1}$ we notice that any geodesic connecting $\pi_{W_{i-1}}(1)$ to $\pi_{W_{i-1}}(C_{i-1})$ lies in the $(M/10+5E)$-neighbourhood of $\pi_{W_{i-1}}(C_{i-1})$, and in particular it cannot pass $E$-close to $\rho^U_{W_{i-1}}$ as 
$$\dist_{W_{i-1}}(C_{i-1}, U)-(M/10+5E)\ge 4M/5-10E\ge 2E;$$
thus the bounded geodesic image Lemma~\ref{lem:bgi} tells us that $\dist_U(1, C_{i-1})\le E$.



Arguing the exact same way, one gets that $\dist_U(w, C_{i+2})\le 3E$, so the picture inside $\C U$ is as in Figure~\eqref{fig:nest_orth_nest}.

\begin{figure}[htp]
\centering
\includegraphics[width=\textwidth]{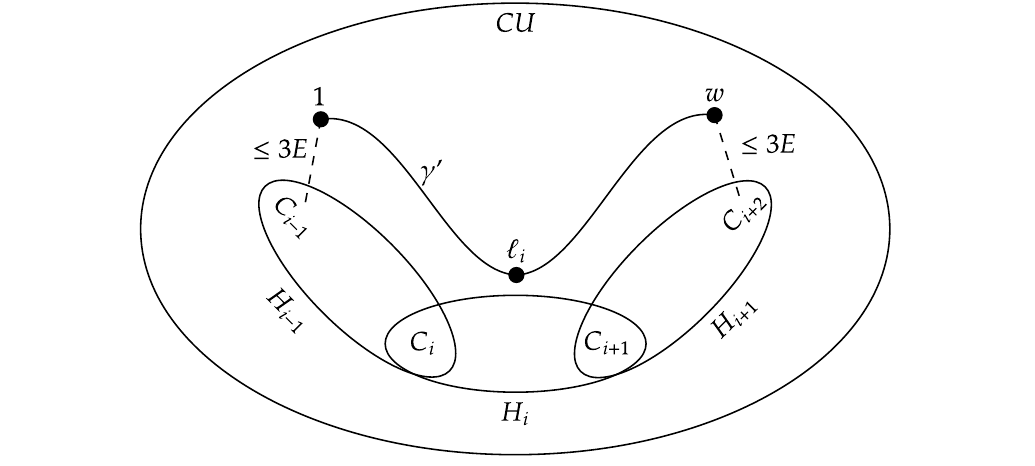}
\caption{The hierarchy path $\gamma'$ projects to a quasigeodesic (after reparametrisation) whose endpoints are $3E$-close to $\pi_U\left(H_{i-1}\cup H_i \cup H_{i+1}\right)$. The latter is a union of quasiconvex subsets of a hyperbolic space, and is therefore quasiconvex.}
\label{fig:nest_orth_nest} 
\end{figure}

Now, the projections of $H_{i-1}$, $H_{i}$, and $H_{i+1}$ inside $\C U$ are all $\kappa(0)$-quasiconvex, so their union is $(\kappa(0)+4E+2)$-quasiconvex by Lemma~\ref{lem:union_of_quasiconvex_in_hyp}. As $\pi_U(\gamma')$ is a $(\lambda', \lambda')$-quasigeodesic (after reparametrisation), whose endpoints $\pi_U(1)$ and $\pi_U(w)$ are within distance at most $3E$ from $\pi_U(C_{i-1})$ and $\pi_U(C_{i+1})$, respectively, we get that $\pi_U(\ell_i)$ must be $\xi$-close to $\pi_U(H_{i-1}\cup H_{i}\cup H_{i+1})$, for some constant $\xi$ depending on $\lambda'$, $\kappa(0)$, and $E$. 

Now, if $\pi_U(\ell_i)$ is $(\xi+T)$-close to $\pi_U(H_{i})$ then we are done. Otherwise, suppose that $\dist_U(\ell_i, H_i)>\xi+T$, so that we can assume without loss of generality that $\pi_U(\ell_i)$ is $\xi$-close to some point $q\in \pi_U(H_{i-1})$. Then by triangle inequality 
$$\diam_U(H_{i-1})\ge\dist_U(C_i,q)\ge \dist_U(C_{i},\ell_i)-\dist_U(\ell_i, q)\ge \dist_U(H_{i}, \ell_i)-\dist_U(\ell_i, q)\ge T.$$
Moreover, $$\diam_{W_i}(H_i)\ge \dist_{W_i}(C_i, C_{i+1})\ge  M\ge T.$$
As $A$ and $B$ fill all squares (Definition~\ref{defn:fill_all_gaps}) and $W_i\orth U$, 
we have that either $\pi_U(\gate_{H_{i-1}}(H_i))$ is $T$-dense in $\pi_U(H_{i-1})$, or $\pi_{W_i}(\gate_{H_{i}}(H_{i-1}))$ is $T$-dense in $\pi_{W_i}(H_{i})$. Combining this with Lemma~\ref{lem:gate=intersection_for_groups}, which states that the gates coarsely coincide with the intersection, and the fact that projection maps are $(E,E)$-coarsely Lipschitz, we get that either $\pi_U(C_i)$ is $(T+LE+E)$-dense in $\pi_U(H_{i-1})$, or $\pi_{W_i}(C_i)$ is $(T+LE+E)$-dense in $\pi_{W_i}(H_{i})$.

However $\dist_{W_i}(C_i, C_{i+1})\ge M>T+LE+E$, again by our choice of $M$; so we must have that $\pi_U(C_i)$ is $(T+LE+E)$-dense in $\pi_U(H_{i-1})$. In turn, this means that
$$\dist_{U}(\ell_i, H_i)\le \dist_U(\ell_i, C_i)\le \dist_U(\ell_i, H_{i-1})+(T+LE+E)\le \xi+T+LE+E,$$
and we are done.
\end{proof}
The proof of Lemma~\ref{lem:nbh_chiuso_per_hp}, and in turn of Theorem~\ref{thm:amalgamation_of_quasiconvex}, is now complete.
\end{proof}



%

\subsection{Why no drift?}\label{subsec:no_drift}
We now show that the conclusion of Theorem~\ref{thm:amalgamation_of_quasiconvex} might not hold if one removes the hypothesis of having no drift in the orthogonals, Definition~\ref{defn:no_drift}. Let $\mathcal F_{a,b}$ be the free group on two generators $a$ and $b$, and let $D_{x,y}\coloneq \langle x,y\,|\,x^2=y^2=1\rangle$ be a copy of $D_{\infty}$ generated by the involutions $x$ and $y$. Let
$$G=\mathcal F_{a,b}\times D_{x_1,y_1}\times D_{x_2,y_2},$$
and let $X$ be the Cayley graph for $G$ with respect to the generators $\{a,b, x_{1}, y_1, x_2, y_2\}$. The $G$-action on $X$, which is a direct product of hyperbolic spaces, makes $G$ into a HHG; in particular, combining \cite[Corollary 8.28 and Theorem 9.1]{HHS_II}, we can find a HHG structure $\frakS$ whose only domains with unbounded coordinate spaces are the following:
\begin{itemize}
    \item For every $g\in \mathcal F_{a,b}$, there is a domain $gL_a$ whose coordinate space is $g\Cay{\langle a\rangle}{a}$, and a domain $gL_b$ defined analogously;
    \item The Bass-Serre tree $T$ of the splitting $\mathcal F_{a,b}=\langle a\rangle * \langle b\rangle$ is a domain, whose coordinate space is the tree itself;
    \item Finally, there are two domains $W_1$ and $W_2$, whose coordinate spaces are, respectively,  $\Cay{D_{x_1,y_1}}{\{x_1,y_1\}}$ and $\Cay{D_{x_2,y_2}}{\{x_2,y_2\}}$.
\end{itemize}
The relations between the above domains are as follows: $T$, $W_1$, and $W_2$ are pairwise orthogonal; for every $g\in \mathcal{F}_{a,b}$, $gL_a$ and $gL_b$ are nested inside $T$; every two domains which are nested inside $T$ are transverse.

Now let $A=\langle a^N x_1x_2\rangle$ and $B=\langle b^N y_1y_2\rangle$, where $N$ is a positive integer to be chosen later. For every $g\in A-\{1\}$ let $Y_g=L_a$, and similarly for every $g\in B-\{1\}$ let $Y_g=L_b$. One can choose $N$ large enough that $A$ and $B$ satisfy the assumptions of Theorem~\ref{thm:amalgamation}. Moreover, notice that the projection of $A$ to every domain which is not $L_a$ has diameter bounded by some constant $K$, while $\pi_{L_a}(A)$ is coarsely dense in $\C L_a$. In particular, $A$ is hierarchically quasiconvex, and similar considerations hold for $B$. We also notice that $A$ and $B$ $(K+1)$-fill all squares, as for every two orthogonal domains $U$ and $V$ we have that $\min\{\diam_U(A),\diam_V(B)\}\le K$. However, $A$ and $B$ do not satisfy Definition~\ref{defn:no_drift}, as for example the unbounded domain $W_1$ is orthogonal to both $L_a$ and $L_b$ but both $A$ and $B$ have bounded projection to $\C W_1$. 

Finally, $A*B$ is not hierarchically quasiconvex. Indeed, the projection of $A*B$ to the product $D_\infty^2\coloneq D_{x_1,y_1}\times D_{x_2,y_2}$ is within finite distance from the diagonal $\langle (x_1y_1,x_2y_2)\rangle$; hence the projection of $A * B$ to both  $W_1$ and $W_2$ is coarsely dense, but a point $p\in D^2_\infty$ can be arbitrarily far from the diagonal. Thus $A*B$ does not satisfy the realisation property from Definition~\ref{defn:HQC}.

\section{Combination of strongly quasiconvex subgroups}\label{sec:combination_strong}
We conclude the paper by studying when our amalgamation procedure preserves the following notion of quasiconvexity:
\begin{defn}
Let $X$ be a geodesic metric space. A subspace $Y\subseteq X$ is $Q$-\emph{strongly quasiconvex}, for some function $Q\colon [0,+\infty)\to [0,+\infty)$ called the \emph{strong convexity gauge} of $Y$, if, given any $\lambda\ge 0$, every $(\lambda, \lambda)$-quasigeodesic with endpoints on $Y$ lies in the $Q(\lambda)$-neighbourhood of $Y$.
\end{defn}
The above notion is equivalent to quasiconvexity in hyperbolic spaces (see e.g. \cite[Theorem III.1.7]{bridsonhaefliger}), but it is stronger in general. 

We also need the following definition from \cite{RST}:
\begin{defn}\label{defn:orth_proj_dich}
    For $\Theta\ge 0$, a subset $A$ of an HHS $(X,\frakS)$ has the \emph{$\Theta$–orthogonal projection dichotomy} if for all $U,V\in\frakS$ with $U\orth V$, if $\diam_U(A)\ge \Theta$ then $\pi_V(A)$ is $\Theta$-dense in $\C V$.
\end{defn}

The following Lemma shows how strong quasiconvexity and hierarchical quasiconvexity are related:

\begin{lem}[{Russell-Spriano-Tran, \cite[Theorem 6.3]{RST}}]\label{lem:russell_spriano_tran}
    Let $(X,\frakS)$ be a hierarchically hyperbolic space. A subspace $Y$ is $Q$-strongly quasiconvex, for some gauge $Q$, if and only if it is $\kappa$-hierarchically quasiconvex and has the $\Theta$-orthogonal projection dichotomy, for some $\kappa$ and $\Theta$. Moreover, the gauge $Q$ and the pair $(\kappa,\Theta)$ each determine the other.
\end{lem}

\begin{thm}\label{thm:amalgamation_of_strongly_quasiconvex_final_version}
 Let $(G,\frakS)$ be a HHG, let $A,B\le G$ be two $Q$-strongly quasiconvex subgroups of $G$, and let $C=A\cap B$. Suppose that there exists $M\ge0$ such that $A$ and $B$ satisfy the hypotheses of Theorem~\ref{thm:amalgamation}, for some choice of $Y_a,Y_b\in \frakS$ for every $a\in A-C$ and every $b\in B-C$.

There exists a positive constant $\mathfrak M\ge 0$ and a function $\mathfrak{Q}\colon[0,+\infty)\to [0, +\infty)$, both depending on $Q$ and $(G,\frakS)$, such that if $M\ge \mathfrak{M}$ then $\langle A,B\rangle_G\cong A*_C B$ is $\mathfrak{Q}$-strongly quasiconvex in $G$.
\end{thm}

In view of Lemma \ref{lem:russell_spriano_tran}, Theorem~\ref{thm:amalgamation_of_strongly_quasiconvex_final_version} can be rephrased in the following form, which is the one we shall prove:

\begin{thm}\label{thm:amalgamation_of_strongly_quasiconvex}
 Let $(G,\frakS)$ be a HHG, let $A,B\le G$ be two $\kappa$-hierarchically quasiconvex subgroups of $G$, and let $C=A\cap B$. Suppose that:
 \begin{itemize}
     \item There exists $M\ge100E$ such that $A$ and $B$ satisfy the hypotheses of Theorem~\ref{thm:amalgamation};
     \item There exists $\Theta\ge 0$ such that $A$ and $B$ have the $\Theta$-orthogonal projection dichotomy (Definition~\ref{defn:orth_proj_dich}).
 \end{itemize}
Then there exist positive constants $\mathfrak M, \mathfrak{T}\ge 0$ and a function $\mathfrak{k}\colon[0,+\infty)\to [0, +\infty)$, all depending on $\kappa$, $\Theta$, and $(G,\frakS)$, such that if $M\ge \mathfrak{M}$ then $\langle A,B\rangle_G\cong A*_C B$ is $\mathfrak{k}$-hierarchically quasiconvex in $G$, and has the $\mathfrak{T}$-orthogonal projection dichotomy.
\end{thm}

\begin{proof}[Proof of Theorem \ref{thm:amalgamation_of_strongly_quasiconvex}]
    Firstly, we prove that, if $M\ge \Theta$, then the orthogonal projection dichotomy for $A$ and $B$ implies the second and third hypotheses of Theorem~\ref{thm:amalgamation_of_quasiconvex} (Definitions~\ref{defn:fill_all_gaps} and~\ref{defn:no_drift}), for some constants $T$ and $R$ depending on $\kappa$, $\Theta$, and $(G,\frakS)$.

\par\medskip
\textbf{$A$ and $B$ fill all squares.} Let $U,V\in\frakS$ be such that $U\orth V$, and suppose that $$\min\{\diam_U(A), \diam_V(B)\}\ge 2\Theta+1.$$
By the orthogonal projection dichotomy, $\pi_U(B)$ is $\Theta$-dense in $\C U$. This means that, for every $a\in A$, there is some $b\in B$ such that $\dist_U(a,b)\le \Theta$, and in particular $\dist_U(a, \gate_A(b))\le 2\Theta+1$ as $\pi_U(\gate_A(b))$ is defined by taking the coarse closest point projection of $\pi_U(b)$ onto $\pi_U(A)$. Hence $\pi_U(\gate_A(B))$ is $(2\Theta+1)$-dense in $\pi_U(A)$, that is, we proved that $A$ and $B$ $T$-fill all squares for $T=2\Theta+1$. 

\par\medskip
\textbf{$A$ and $B$ have no drift in the orthogonals.} Next, notice that the $\Theta$-orthogonal projection dichotomy for $A$ and $B$ implies that $A$ and $B$ have no drift in the orthogonals, for $R=\Theta$. Indeed, if $b\in B-C$ and $U\in\frakS$ is orthogonal to $Y_b$, then $\pi_U(B)$ is $\Theta$-dense in $\C U$, as we know that $\diam_{Y_b}(B)\ge \dist_{Y_b}(C,bC)\ge M\ge\Theta$.

\par\medskip
\newcommand{\DIAMETER}{100\Theta+100E+4R+1}
\textbf{Orthogonal projection dichotomy.} Now assume that 
$$M\ge \Theta+4E+\mathfrak M,$$
where $\mathfrak M$ is the constant from Theorem~\ref{thm:amalgamation_of_quasiconvex} (which in turn depends on $\kappa$, $E$, and the constants $T$ and $R$ from the previous paragraphs). This choice of $M$ grants the hierarchical quasiconvexity of $A*_C B$. Our final goal is to prove that $A*_C B$ satisfies the $\mathfrak{T}$-orthogonal projection dichotomy, where \begin{equation}\label{equation:mathfrakT}
    \mathfrak{T}=\DIAMETER.
\end{equation}

Let $U\orth V$ be such that $\diam_U(A*_C B)\ge \mathfrak{T}$. Up to the action of the group, we can assume that $\dist_U(1,w)\ge \mathfrak{T}/2$, where $w=g_1\ldots g_k c$ is some element of $A*_C B$, with every $g_i$ in a different factor than $g_{i+1}$. Define $H_i$, $C_i$, and $W_i$ as in Notation~\ref{notation:ACW}.

If $\diam_U(H_i)\ge \Theta$ for some $i$ then, by the orthogonal projection dichotomy for either $A$ or $B$, we have that $\pi_V(H_i)$ is $\Theta$-dense in $\C V$. In particular $\pi_V(A*_C B)$ is $\mathfrak{T}$-dense in $\C V$, and we are done. 

\par\medskip
Thus suppose that $\max_{i=1, \ldots, k}\{\diam_U(H_i)\}< \Theta$. Our current goal is to find an index $j$ such that, inside $U$, the factor $H_{r}$ projects far from both $1$ and $w$ whenever the difference $|j-r|$ is sufficiently small, as we will clarify in Equation~\eqref{eq:far_from_1w} below.

Let $j\le k$ be the first index for which $$\dist_U(1, H_j)\ge \mathfrak T/4-\Theta$$ (such an index exists as, for example, $\dist_U(1, H_k)\ge \dist_U(1, w)-\diam_U(H_k)\ge \mathfrak T/2-\Theta$). Notice that $j>3$, as $1\in H_1$ and $\diam_U(H_1\cup H_2\cup H_3)\le 3\Theta$, which is strictly less than $\mathfrak T/4-\Theta$ by our choice of $\mathfrak T$, Equation~\eqref{equation:mathfrakT}. Furthermore $\dist_U(1, H_{j-1})\le \mathfrak T/4-\Theta$, by minimality of $j$. Hence
$$\dist_U(H_j, w)\ge \dist_U(1, w)-\dist_U(1, H_{j-1})-\diam_U(H_{j-1}\cup H_j)\ge \mathfrak T/4-3\Theta,$$
where we used that $H_{j-1}$ intersects $H_j$ and therefore $\diam_U(H_j\cup H_{j-1})\le 2\Theta$. Summarising, we found an index $j$ such that
$$\min\{\dist_U(1, H_j), \dist_U(H_j,w)\}\ge \mathfrak T/4-3\Theta.$$
Now notice that, whenever $|j-r|\le 3$, we have that 
\begin{equation}\label{eq:far_from_1w}
    \min\{\dist_U(1, H_{r}), \dist_U(H_{r},w)\}\ge \mathfrak T/4-7\Theta.\end{equation}
Indeed
$$\dist_U(1, H_{r})\ge\dist_U(1, H_{j})-\diam_U\left(\bigcup_{t=j}^{r}H_t\right),$$
and as each $H_t$ intersects the next one we have that
$$\diam_U\left(\bigcup_{t=j}^{r}H_t\right)\le \sum_{t=j}^{r}\diam_U(H_r)\le (|j-r|+1)\Theta\le 4\Theta.$$ 
The same argument shows that $\dist_U(H_{r},w)\ge \mathfrak T/4-7\Theta$, hence the situation in $\C U$ is as in Figure~\eqref{fig:block_of_close_Ai}.

\begin{figure}[htp]
    \centering
    \includegraphics[width=\textwidth]{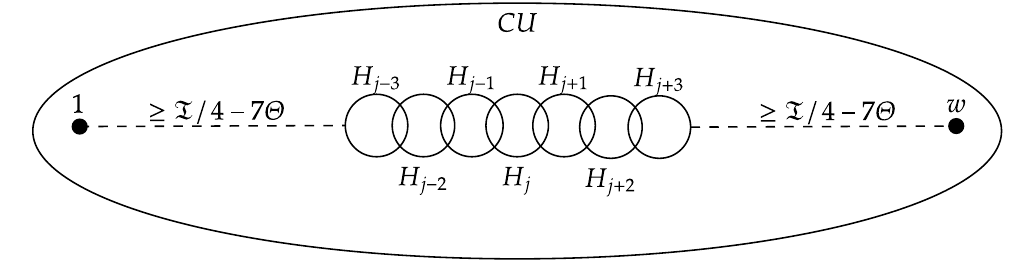}
    \caption{Every $\pi_U(H_i)$ has diameter at most $\Theta$, and intersects non-trivially both $\pi_U(H_{i-1})$ and $\pi_U(H_{i+1})$. Hence there is a “chain” of $H_i$s connecting $1$ to $w$. In the picture, we selected a collection of seven $H_i$s whose projection is far from both $1$ and $w$.}
    \label{fig:block_of_close_Ai}
\end{figure}

\par\medskip
Now we look at the relation between $U$ and any $W_r$ for which $|j-r|\le 3$. 

\par\medskip
\textbf{Containment implies the conclusion.} Suppose first that $W_r\nest U$, for some $W_r$ as above. Then $V$, which was orthogonal to $U$, is also orthogonal to $W_r$. As $$\diam_{W_r}(H_r)\ge \dist_{W_r}(C_r, C_{r+1})\ge M\ge \Theta,$$ the orthogonal projection dichotomy for $H_r$ implies that $\pi_V(H_r)$ is $\Theta$-dense in $\C V$, and we are done. Then from now on assume that $U$ does not contain any $W_r$ as above.

\par\medskip
\textbf{No transversality.} Next, we argue that $U$ cannot be transverse to any $W_r$. Indeed $\dist_{W_r}(C_r, C_{r+1})\ge  M\ge 4E$. Therefore, if $U$ were transverse to $ W_r$, then at least one between $C_r$ and $C_{r+1}$ would be at distance greater than $E$ from $\rho^U_{W_r}$. By the Behrstock inequality, this would imply that $\dist_U(W_r, H_r)\le E$, and in turn
$$\dist_U(1, W_r)\ge\dist_U(1, H_{r})-\diam_{U}(\rho^{W_r}_U)-\dist_U(W_r, H_r)\ge$$
$$\ge\mathfrak{T}/4-7\Theta-2E\ge 2E,$$
by how we chose $\mathfrak{T}$ in Equation~\eqref{equation:mathfrakT}. The same proof gives that $\dist_U(W_{r},w)\ge 2E$. But then the Behrstock inequality would again imply that $1$ and $w$ both project $E$-close to $\rho^U_{W_r}$ inside $W_r$, which is impossible.

\par\medskip
\textbf{No consecutive orthogonality.} Then we can assume that $U$ is either orthogonal to, or nested in, each $W_r$ with $|j-r|\le 3$. 
Suppose that both $|j-r|\le 3$ and $|j-(r-1)|\le 3$. We claim that $U$ cannot be orthogonal to both $W_{r-1}$ and $W_r$. If this was the case, then $\pi_U(C_r)$ would be $R$-dense inside $\C U$, as $A$ and $B$ have no drift in the orthogonals (Definition~\ref{defn:no_drift}). But this cannot happen, as $\diam_U(C_r)\le\diam_U(H_r)\le \Theta$ while $\diam(\C U)\ge \dist_U(1,w)\ge \mathfrak{T}/2$, which is strictly greater than $2R+\Theta$ by our choice of $ \mathfrak{T}$ in Equation~\eqref{equation:mathfrakT}.

\par\medskip
\textbf{No alternated nesting.} We are left with the case when, whenever $|j-r|\le 3$ and $|j-(r-1)|\le 3$, $U$ is nested inside one between $W_{r-1}$ and $W_r$. As there are seven domains between $W_{j-3}$ and $W_{j+3}$, we can find three indices $j-3\le a<b<c\le j+3$ such that $U$ is nested in $W_a$, $W_b$, and $W_c$. The consistency Axiom~\eqref{item:dfs_transversal} yields that $\dist_{W_b}(W_a, U)\le E$ and $\dist_{W_b}(W_c, U)\le E$. Hence
$$\dist_{W_b}(W_a,W_c)\le \dist_{W_b}(W_a, U)+\diam(\rho^U_{W_b})+\dist_{W_b}(W_a, U)+\le 3E.$$
This gives a contradiction, because $a<b<c$ and Claim~\ref{claim:distance_of_consecutives} tells us that $\dist_{W_b}(W_a,W_c)$ must be at least $6E$.
\par\medskip
The proof of Theorem~\ref{thm:amalgamation_of_strongly_quasiconvex} is now complete.
\end{proof}

\bibliography{biblio}
\bibliographystyle{alpha}
\end{document}